\documentclass[a4paper, 10pt]{article}
\usepackage[utf8]{inputenc}
\usepackage{color}

\usepackage{amsmath} 
\usepackage{amssymb} 
\usepackage{amsthm}
\usepackage{graphicx}
\usepackage[colorlinks=true,linkcolor=blue]{hyperref}
\usepackage{hyphenat}

\theoremstyle{plain}

\newtheorem{prop}{Proposition}[section]

\newtheorem{lem}{Lemma}[section] \newtheorem{cor}{Corollary}[section]

\newtheorem{defi}{Definition}[section]

\theoremstyle{remark}
\newtheorem{rmk}{Remark}

\newcommand {\R} {\mathbb{R}} \newcommand {\Z} {\mathbb{Z}}
 \newcommand {\N} {\mathbb{N}}

\newcommand {\p} {\partial}
\newcommand {\D} {\Delta}

\newcommand {\dt} {\partial_t}

\newcommand{\va}{\varphi}

\DeclareGraphicsExtensions{.jpg,.pdf,.pdftex,.eps,.png, .ps}

\newcommand {\Ds} {(-\Delta)^{s}}
\newcommand {\drr} {\partial_r}

\newcommand {\e} {\eta_{\delta,r}}
\newcommand {\E} {E_{\lambda}}
\newcommand {\Et} {\tilde{E}_{\lambda}}
\DeclareMathOperator {\supp} { supp}
\DeclareMathOperator {\dist} { dist}

\DeclareMathOperator {\spec} { spec}

\newcommand {\dv} { \frac{\partial}{\partial \varphi}}

\DeclareMathOperator{\la}{l}
\DeclareMathOperator{\lb}{\bar{l}}

\DeclareMathOperator{\tr}{tr}

\begin{document}

\title{Unique Continuation for Fractional Schr\"odinger Equations with Rough Potentials\footnote{This work is part of the PhD thesis of the author written under the supervision of Prof. Herbert Koch to whom she owes great gratitude for his persistent support and advice. She thanks the Deutsche Telekomstiftung and the Hausdorff Center for Mathematics for financial support. }}
\author{Angkana R\"uland\footnote{Mathematisches Institut, Universit\"at Bonn, Endenicher Allee 60, 53115 Bonn, Germany, rueland@math.uni-bonn.de.}}
\maketitle

\begin{abstract}
This article deals with the weak and strong unique continuation principle for fractional Schrödinger equations with scaling-critical and rough potentials via Carleman estimates. Our methods allow to apply the results to ``variable coefficient" versions of  fractional Schrödinger equations. 
\end{abstract}

\tableofcontents

\section{Introduction}

%\subsection{Main Results}
\label{sec:resultsfl}

We consider the strong unique continuation problem (SUCP) for (weak solutions of) fractional Schrödinger equations, i.e. $u \in H^{s}$, $s\in(0,1)$, satisfies
\begin{align*}
(-\D)^{s} u = Vu \mbox{ in } \R^n,
\end{align*}
(in a weak sense) and vanishes of infinite order at the origin. We prove that under appropriate conditions on $V$ (including scaling-critical Hardy potentials) the solution $u$ vanishes identically:

\begin{prop}[SUCP]
Let $u\in H^{s}$, $s\in(0,1)$, solve
\begin{align}
\label{eq:frac}
(-\D)^{s} u = Vu \mbox{ in } \R^n,
\end{align}
with $V= V_{1} + V_{2}$,
\begin{align*}
V_{1}(y) = |y|^{-2s}h\left(\frac{y}{|y|}\right), \ \  h\in L^{\infty} , \ \ |V_{2}(y)| \leq c|y|^{-2s+\epsilon}. %, \ \ \left\|h \right\|_{L^{\infty}} \ll 1.
\end{align*}
For $s < \frac{1}{2}$, we additionally require that one of the following assumptions is satisfied:
\begin{itemize}
\item the potential $V_2$ satisfies
$
V_{2}\in C^{1}(\R^n\setminus \{0\}) \mbox{ and }
|y\cdot \nabla V_{2}| \lesssim c|y|^{-2s+\epsilon},
$
\item $s\in [\frac{1}{4}, \frac{1}{2})$ and $V_1 \equiv 0$.
\end{itemize}
Then if $u$ vanishes of infinite order at $y=0$, this already implies $u\equiv 0$.
\end{prop}

Here the infinite order of vanishing is adapted to the degenerate elliptic equation derived via the Caffarelli extension. We define the notions of vanishing of infinite order for bulk and for corresponding boundary integrals.

\begin{defi}[Vanishing of Infinite Order]
\label{defi:vio}
A function $u\in L^{2}_{loc}(y_{n+1}^{1-2s}dy,\R^{n+1}_{+})$ \emph{vanishes of infinite order at zero (in the bulk)} if for every $m\in\N$ 
\begin{align*}
\lim\limits_{r\rightarrow 0}r^{-m}\int\limits_{B_{r}^+(0)} y_{n+1}^{1-2s}u^2 dy = 0.
\end{align*}
A function $u\in L^{2}_{loc}(\R^n)$ \emph{vanishes of infinite order at zero (at the boundary)} if for every $m\in \N$
\begin{align*}
\lim\limits_{r\rightarrow 0}r^{-m}\int\limits_{B_{r}(0)} u^2 dy = 0.
\end{align*}
\end{defi}

Recently, the problem of weak and strong unique continuation for the fractional Laplacian has received a certain amount of attention: The \emph{weak} unique continuation case is treated by Seo \cite{Seo}, \cite{Seo1}. In \cite{Seo} Seo argues via expansions of the corresponding Green's function kernels and Carleman-like estimates under the assumptions of $u\in L^1$, $Vu\in L^1$ and $n-1\leq 2s \leq n$ while he  employs techniques related to the article of \cite{JK} which allow to extend the weak unique continuation statement to scaling-critical spaces in \cite{Seo1}.\\  %comment on Seo's new paper
The furthest result concerning \emph{strong} unique continuation for certain scaling-critical Hardy potentials with properly signed or sufficiently small multiplicative prefactors is treated by Fall and Felli \cite{FF} via frequency function methods. A crucial point in their strategy is to localize the problem with the aid of the Caffarelli-Silvestre extension which allows to construct an appropriate notion of frequency function.\\

In the present work, we approach the problem via Carleman inequalities. We argue in two main steps:
\begin{itemize}
\item \emph{Carleman Estimates.} This part constitutes the key estimate: We prove a Carleman inequality at the boundary of the upper half-plane. 
Our strategy of proving the decisive Carleman estimate relies on methods of Koch and Tataru \cite{KT1}. We separate the conjugated operator into a radial and a spherical part. Then, we decompose the spherical operator into its eigenspaces. Thus, the necessary estimate is reduced to a bound on (the kernel of) an ordinary differential operator. This procedure allows to handle very rough potentials for $s\geq \frac{1}{4}$ (including scale-invariant ones for $s\geq \frac{1}{2}$ and subcritical ones for $s\geq \frac{1}{4}$). 
If $s\in(0,\frac{1}{4})$ (or if $s\in [\frac{1}{4}, \frac{1}{2})$ and $V$ includes scaling-invariant potentials), we argue with the help of a slightly modified Carleman estimate which allows to exploit the differentiability assumptions in order to deduce the unique continuation property (c.f. Section \ref{sec:cor}). \\
In the case of a sufficiently strong spectral gap, e.g. as in the one-dimensional situation, the unique continuation property can be deduced for any potential which is bounded by a Hardy type potential, $|V(y)|\leq c|y|^{-2s}$ (c.f. Section \ref{sec:1D}) if $s>\frac{1}{2}$ (for $s= \frac{1}{2}$ an additional smallness condition has to be satisfied: $0<c\ll1$).
\item \emph{Blow-up Procedure.} With the previously discussed preparation, it becomes possible to conclude that if the Caffarelli extension vanishes of infinite order in the tangential and normal directions (with respect to the boundary), it must already vanish identically. Hence, we can concentrate on extensions which only vanish of finite order in the normal direction. For these we consider a blow-up procedure which reduces the problem to the weak unique continuation property. %comment on the case $s\leq \frac{1}{4}$
\end{itemize}

Our approach does not only complement the previous literature by relying on Carleman instead of frequency function methods. It also improves a number of results including the following three main aspects:
\begin{itemize}
\item In the case of the \emph{one-dimensional} situation and $s\geq \frac{1}{2}$, it is possible to treat arbitrary potentials which are bounded by scaling-critical Hardy type potentials (with a smallness condition for $s=\frac{1}{2}$). % with a smallness condition on the associated prefactor. 
This is a consequence of the explicit estimates on the spectral gap of the extension operator (c.f. Section \ref{sec:1D}).
\item We allow for \emph{arbitrarily large scaling-critical potentials}. Furthermore, our subcritical potentials need not be differentiable. 
In the frequency function framework a regularity restriction was needed in order to deduce Pohozaev identities.
\item Our approach allows for generalizations to \emph{variable coefficient problems}. In this sense we can treat ``variable coefficient" fractional Laplacian operators, c.f. Section \ref{sec:vc}.
\end{itemize}

If $s\geq \frac{1}{4}$, our main results are derived as consequences  of the following Carleman estimate:

\begin{prop}[Symmetric Carleman Estimate] % right potential!
\label{prop:KT1b}
Let $s\in[\frac{1}{4},1)$ and let $$\phi(y) = - \ln(|y|) + \frac{1}{10}\left(\ln(|y|)\arctan(\ln(|y|))- \frac{1}{2}\ln(1+\ln(|y|)^2) \right).$$ Consider $w\in H^{1}(y_{n+1}^{1-2s}dy,\R^{n+1}_+)$ with 
\begin{align*}
\nabla \cdot y_{n+1}^{1-2s} \nabla  w & = f \mbox{ in } \R^{n+1}_{+}, \\
\lim\limits_{y_{n+1}\rightarrow 0} y^{1-2s}_{n+1} \p_{n+1} w &= h \mbox{ on } \R^n.
\end{align*}
Then for $\tau \geq \tau_0 > 0$ we have
\begin{equation}
\label{eq:Carlflsymb}
\begin{split}
& \left\| e^{\tau \phi}\left(1+\ln(|y|)^2\right)^{-\frac{1}{2}}y_{n+1}^{\frac{1-2s}{2}} \nabla w \right\|_{L^2(\R^{n+1}_{+} )}  \\ 
&+ \tau \left\| e^{\tau \phi}\left(1+\ln(|y|)^2\right)^{-\frac{1}{2}}y_{n+1}^{\frac{1-2s}{2}}|y|^{-1} w \right\|_{L^{2}(\R^{n+1}_{+})}\\
&+ \tau^{s} \left\| e^{\tau \phi}\left(1+\ln(|y|)^2\right)^{-\frac{1}{2}}|y|^{-s} w \right\|_{L^2(\R^{n})}\\
\lesssim & \ \tau^{- \frac{1}{2}} \left\|e^{\tau \phi} |y| y_{n+1}^{\frac{2s-1}{2}} f\right\|_{L^2(\R^{n+1}_{+})} 
\ + \tau^{\frac{1-2s}{2}}  \left\| e^{\tau \phi} |y|^{s} h \right\|_{L^{2}(\R^{n})}.
\end{split}
\end{equation}
\end{prop}

We remark that in the case of the half-Laplacian our results can be sharpened by using the framework established by Koch and Tataru \cite{KT1} dealing with equations of the following form:
\begin{align}
\label{eq:KT1}
\p_{i} g^{ij} \p_{j} u = V u + W_{1} \nabla u + \nabla W_{2} u,
\end{align}
where $V\in c_{0}(L^{\frac{n}{2}})$, $W_{1}, W_{2}\in l^{1}_{w}(L^{n})$ (the function spaces are built by a dyadic summation over annuli) and where $g^{ij}$ are Lipschitz perturbations of the Laplacian. 
Our problem can be phrased in a similar strong unique continuation framework for (degenerate) elliptic operators by considering the evenly reflected Caffarelli extension. In this case we obtain an equation of the form (\ref{eq:KT1}) where $g^{ij} = |y_{n+1}|^{1-2s} id$ is now degenerate (unless $s= \frac{1}{2}$), $V = 0$ and $W_{1} = (0,...,0,H(y_{n+1}))W(y')$, $W_{2} = H(y_{n+1})W(y')$ are Heaviside functions at the boundary. Hence, in the case of the half-Laplacian, the strong unique continuation problem can directly be treated with the methods of Koch and Tataru if $V\in l_{w}^1(L^{n+1})$. Via an improved extension, c.f. Section {\ref{sec:Half}}, we show that this still remains true for $V\in L^{n+\epsilon}$. For the general fractional Laplacian it appears to be more difficult to reduce the integrability requirements on the potentials via similar means, since the symmetric operator in the Carleman estimates does not yield sufficiently strong positivity anymore. \\

Last but not least, we would like to stress that our strategy does not only apply to the fractional Laplacian but also works for a much larger class of operators. For any boundary value problem such that the underlying operator 
\begin{itemize}
\item allows for a sufficiently strong Carleman inequality at the boundary,
\item allows for sufficiently strong boundary estimates,
\end{itemize}
our strategy can be used to derive the strong unique continuation property.\\

Let us finally comment on the structure of the remaining article. In the following Section we recall the Caffarelli-Silvestre extension and present an argument for the weak unique continuation principle for the (localized version of the) fractional Laplacian. In Section \ref{sec:subsymm} we prove the main Carleman estimate, Proposition \ref{prop:KT1b}. This is then used to deduce doubling estimates and the strong unique continuation statement in Section \ref{sec:dewucp}. In Sections \ref{sec:1D}, \ref{sec:Half} we focus on refinements of our statement in the one-dimensional setting and in situations involving the half-Laplacian. Last but not least, Section \ref{sec:vc} deals with generalizations to ``variable coefficient" Laplacian operators and to operators on perturbations of flat domains.

\section{Setting and Weak Unique Continuation Principle}
In the sequel we recall the definition of the fractional Laplacian via its Caffarelli-Silvestre extension. We indicate how this can be used to prove the weak unique continuation principle. 

\subsection{The Caffarelli-Silvestre Extension}
The fractional Laplacian can be defined via several methods including its Fourier representation, its principal value integral or via extension. In the sequel we will mainly rely on the extension property which we briefly recall.\\

In their celebrated paper, \cite{CaS}, Caffarelli and Silvestre extend the local interpretation  of the half-Laplacian to the whole range of fractional Laplacian operators. They point out that the solutions of
$(-\D)^s u$ for $u\in H^{s}(\R^n)$ can be interpreted as the $H^{-s}(\R^n)$ limit of a generalized Dirichlet-to-Neumann map associated with the (weak) ``harmonic" extension of $u$. More precisely, let $u\in H^{s}(\mathbb{R}^n)$ and consider the extension problem:
\begin{align*}
\nabla \cdot y_{n+1}^{1-2s} \nabla \tilde{u} &= 0 \mbox{ in } \R^{n+1}_{+},\\
\tilde{u} & = u \mbox{ on } \R^n.
\end{align*}
Then $(-\D)^s u=-c_s \lim\limits_{y_{n+1}\rightarrow 0} y_{n+1}^{1-2s}\p_{n+1} \tilde{u}$, where $c_s$ is a constant which only depends on the parameter $s$. With a slight abuse of notation, we will omit the constant in the sequel.

As Fall and Felli \cite{FF}, we will work in this setting as it allows to use local arguments. 

\subsection{The Weak Unique Continuation Property}

\label{sec:weak}

As a first step towards the strong unique continuation result for the fractional operator, we recall the weak unique continuation property for the fractional Laplacian. 

\begin{prop}[Weak Unique Continuation]
\label{lem:wuc}
Let $s\in(0,1)$ and let $u:\R^n\rightarrow \R$, $u\in H^{s}(\R^n)$, solve
\begin{equation*}
\begin{split}
\Ds u &= Vu \mbox{ on } \R^n, \\
u&=0 \mbox{ on } \R^n\cap B_{1}(0).
\end{split}
\end{equation*}
Then $u\equiv 0$.
\end{prop}

Although this property follows from the work of Fall and Felli, c.f. \cite{FF}, by considering the case $V=0$, we provide an argument for it and strengthen the result to a local statement on the Caffarelli-Silvestre extension. More precisely, we show:

\begin{prop}
\label{prop:wuc}
Let $s\in(0,1)$ and let $\tilde{u}\in H^{1}_{loc}(y_{n+1}^{1-2s}dy,B_{1}^+(0))\cap L^{\infty}(B_{1}^+(0))$ solve
\begin{equation}
\begin{split}
\label{eq:CF}
\nabla \cdot y_{n+1}^{1-2s} \nabla \tilde{u} & = 0 \mbox{ in } B_{1}^+(0),\\
\lim\limits_{y_{n+1}\rightarrow 0} y_{n+1}^{1-2s} \p_{n+1} \tilde{u} & = 0 \mbox{ on } B_{1}^+(0)\cap \{y_{n+1}=0\}.
\end{split}
\end{equation}
Further, assume that $\tilde{u}(y',0)=0 \mbox{ on } B_{1}^+(0)\cap \{y_{n+1}=0\}$. Then $\tilde{u}\equiv 0$ in $B_{1}^+(0)$.
\end{prop}

In order to see this, we make use of the equation and regularity estimates for the Caffarelli-Silvestre extension. These ingredients can be combined in a boot strap argument.

\begin{proof}
We first point out the following two facts:
\begin{itemize}
\item[1.] The regularity theory for $H^{1}_{loc}(y_{n+1}^{1-2s}dy,B_{1}^+(0))\cap L^{\infty}(B_{1}^+(0))$ weak solutions of
\begin{equation}
\begin{split}
\label{eq:CF1}
\nabla \cdot y_{n+1}^{1-2s} \nabla \tilde{u} & = 0 \mbox{ in } B_{1}^+(0),\\
\lim\limits_{y_{n+1} \searrow 0} y_{n+1}^{1-2s} \p_{n+1}\tilde{u} & = f \mbox{ on } B_{1}^+(0)\cap \{ y_{n+1}=0 \},
\end{split}
\end{equation}
implies that if $f\in C^{0,\alpha}(B_{1}(0)\cap \{y_{n+1}=0\})$, then $y_{n+1}^{1-2s} \p_{n+1} \tilde{u} \in C^{0,\beta}(\overline{B_{\frac{3}{4}}^+(0)})$ and 
\begin{align*} 
&\left\| y_{n+1}^{1-2s}\p_{n+1}\tilde{u} \right\|_{C^{0,\beta}(\overline{B_{\frac{3}{4}}^+(0)})} \leq C_{1},
\end{align*}
with $C_{1}=C_{1}(s,n,\left\| f\right\|_{L^{\infty}(B_{1}\cap \{y_{n+1}=0\})}, \left\| f \right\|_{C^{0,\alpha}(B_{1}(0)\cap \{y_{n+1}=0\})})$. This follows, for example, from the article by Cabr\'e and Sire \cite{CSi}.
\item[2.] For $a\in(-\infty,1)$ the mean value theorem and the fundamental theorem of calculus imply that for $u\in C^{1}((0,1))\cap C^{0}([0,1))$ the assumptions
\begin{align*}
u(0)=0 \mbox{ and } \lim\limits_{y \searrow 0} y^a u'(y) = 0,
\end{align*}
result in $\lim\limits_{y \searrow 0} y^{a-1}u (y)=0$.
\end{itemize}

Step 1: Beginning of the Iteration.
We make use of the equation: For this we note that the boundary conditions in (\ref{eq:CF}) allow to carry out an even reflection and interpret the solution as a $H^{1}_{loc}(|y_{n+1}|^{1-2s}dy, B_{1}(0))\cap L^{\infty}(B_{1}(0))$ solution of 
\begin{equation}
\begin{split}
\label{eq:CFt}
\nabla \cdot |y_{n+1}|^{1-2s} \nabla \tilde{u} & = 0 \mbox{ in } B_{1}(0).
\end{split}
\end{equation}
For some $\alpha \in (0,1)$ it is $C^{0,\alpha}$-regular (in any direction) and $C^{\infty}$-smooth in the tangential directions \cite{CaS} (quantitative estimates follow, for example, by carrying out a tangential Fourier transform and treating the remaining equation as an ODE in the normal variable). Thus, it is possible to differentiate (\ref{eq:CFt}) with respect to the tangential directions up to an arbitrary order. Using the continuity of, for instance, $ |y_{n+1}|^{1-2s} \p_{n+1} \D' \tilde{u}$ (in $B_{\frac{3}{4}}(0)$) and recalling the even reflection, we obtain
\begin{equation}
\label{eq:tang}
\lim\limits_{y_{n+1}\searrow 0} y_{n+1}^{1-2s} \p_{n+1} \D' \tilde{u}=0.
\end{equation}
By the second preliminary remark from above, this leads to 
\begin{align}
\label{eq:tang1}
\lim\limits_{y_{n+1}\searrow 0} y_{n+1}^{-2s} \D' \tilde{u}=0 \mbox{ and } y_{n+1}^{-2s}\D' u \in C^{0,\gamma}(B_{\frac{3}{4}}(0)).
\end{align}
Hence, we can employ equation (\ref{eq:CF}) to deduce
\begin{align*}
\lim\limits_{y_{n+1}\searrow 0} \p_{n+1} y_{n+1}^{1-2s} \p_{n+1} \tilde{u} = - \lim\limits_{y_{n+1}\searrow 0} y_{n+1}^{1-2s}\D' \tilde{u} = 0.
\end{align*} 
For later use, we highlight that this implies 
$$\lim\limits_{y_{n+1}\searrow 0}y_{n+1}^{-2s}\p_{n+1}\tilde{u}=\lim\limits_{y_{n+1}\searrow 0}y_{n+1}^{-2s-1}\tilde{u}=0.$$

Step 2: Iteration.
With the previous considerations, it is possible to differentiate (\ref{eq:CF}) in the $y_{n+1}$-direction and consider a weak solution of
\begin{equation}
\begin{split}
\label{eq:CF2}
\D  (y_{n+1}^{1-2s} \p_{n+1} \tilde{u}) & = -(1-2s)y_{n+1}^{-2s}\D' \tilde{u} \mbox{ in } B_{\frac{3}{4}}^+(0),\\
\lim\limits_{y_{n+1}\searrow 0} \p_{n+1} (y_{n+1}^{1-2s} \p_{n+1} \tilde{u}) & = 0 \mbox{ on } B_{\frac{3}{4}}^+(0)\cap \{y_{n+1}=0\}.
\end{split}
\end{equation}
Using the observations (\ref{eq:tang}) and (\ref{eq:tang1}), this leads to
\begin{align*}
\lim\limits_{y_{n+1}\searrow 0} \p_{n+1}^2 y_{n+1}^{1-2s}\p_{n+1}\tilde{u} =& \ -(1-2s)\lim\limits_{y_{n+1}\searrow 0} y_{n+1}^{-2s}\p_{n+1} \tilde{u}\\
& \ - \lim\limits_{y_{n+1}\searrow 0} y_{n+1}^{1-2s}\p_{n+1}\D' \tilde{u} = 0.
\end{align*}
Therefore,
\begin{align*}
\lim\limits_{y_{n+1}\searrow 0} \p_{n+1}^2 y_{n+1}^{1-2s} \p_{n+1} \tilde{u} =\lim\limits_{y_{n+1}\searrow 0}y_{n+1}^{-2s-2} \tilde{u}=0.
\end{align*}
As before, we need to complement this by limiting behaviour of tangential derivatives in order to estimate the contributions in the new right hand sides of a differentiated version of (\ref{eq:CF2}). We obtain this by reflecting the function $w(y):=y_{n+1}^{1-2s}\p_{n+1}\tilde{u}$ evenly onto the whole unit ball. In analogy to the previous considerations from step 1, it solves an equation of the type (\ref{eq:CF2}) in the whole unit ball. We differentiate in the tangential directions. For instance, if we consider second tangential derivatives, this implies the continuity of $\p_{n+1} \D' w$, which then results in $\lim\limits_{y_{n+1}\searrow 0}\p_{n+1} \D' w=0$ (for this we also use higher order analogues of (\ref{eq:tang1}) which follow from taking higher order tangential derivatives in step 1). By virtue of the second remark from above and the definition of $w$, this implies 
\begin{align*}
\lim\limits_{y_{n+1}\searrow 0} \p_{n+1}y_{n+1}^{1-2s}\p_{n+1} \D' \tilde{u} = \lim\limits_{y_{n+1}\searrow 0}y_{n+1}^{-2s}\p_{n+1} \D' \tilde{u} = \lim\limits_{y_{n+1}\searrow 0} y_{n+1}^{-2s-1} \D' \tilde{u} =0.
\end{align*}
These terms, however, exactly form the right hand side contributions which result from differentiating (\ref{eq:CF2}) in the normal direction once more. Thus, a bootstrap argument is possible.\\

Step 3: Conclusion. Using the bootstrap procedure, we obtain
\begin{align*}
\lim\limits_{y_{n+1}\searrow 0} y_{n+1}^{-m} \tilde{u} = 0 
\end{align*}
for all $m\in \N$, i.e. $\tilde{u}$ vanishes of infinite order in the normal direction at $y=0$. Combined with the vanishing in the tangential direction and the Carleman inequality from Proposition \ref{prop:KT1b}, this yields $u\equiv 0$ in $B_{1}^+(0)$. 
\end{proof}

\begin{rmk}
If $s=\frac{1}{2}$, the statement of the proposition follows from the weak unique continuation property of the $(n+1)$-dimensional Laplacian. This can be seen by extending the Caffarelli-Silvestre extension, $\tilde{u}$, trivially in the negative $y_{n+1}$-direction. 
\end{rmk}

\section{Symmetric Carleman Estimates}
\label{sec:subsymm}

\subsection{Conformal Coordinates}
\label{sec:CarlII}
In order to prove the desired Carleman inequality, we carry out a change of coordinates similar as in \cite{KT1}.\\

Starting from polar coordinates, the degenerate elliptic operator $\nabla\cdot y_{n+1}^{1-2s} \nabla$ reads
\begin{align*}
\theta_{n}^{1-2s}\frac{1}{r^n}\drr(r^{n+1-2s}\drr) + r^{-1-2s}\nabla_{S^{n}}\cdot \theta_{n}^{1-2s}\nabla_{S^{n}} ,
\end{align*}
where $\theta_{n}= \frac{y_{n+1}}{|y|}=\sin(\va)$.
We transform into conformal coordinates, i.e. $r=e^{t}$, which yields $\drr = e^{-t}\dt$. This leads to 
\begin{align*}
e^{-(1+2s)t}\left[ \theta_{n}^{1-2s} \dt^2 + (n-2s)\theta_{n}^{1-2s}\dt + \nabla_{S^{n}}\cdot \theta_{n}^{1-2s}\nabla_{S^{n}} \right].
\end{align*}
Conjugating with $e^{-\frac{n-2s}{2}t}$ (which corresponds to setting $w=e^{- \frac{n-2s}{2}t}u$) and multiplying the operator with $e^{(1+2s)t}$, results in
\begin{align}
\label{eq:op1}
\theta_{n}^{1-2s}\left(\dt^2 - \frac{(n-2s)^2}{4}\right) + \nabla_{S^{n}}\cdot \theta_{n}^{1-2s} \nabla_{S^{n}}.
\end{align}
In the case of $s= \frac{1}{2}$ this corresponds to the situation in \cite{KT1}.\\

In the sequel we will be using several changes of coordinates. In order to avoid confusion, we clarify the conventions we will be adhering to:
\begin{rmk}[Notation]
\label{rmk:not}
In the proof of Proposition \ref{prop:KT1b} (and in the remaining text) 
\begin{itemize}
\item we use $w$ to denote the original function in Cartesian variables,
\item after a change to conformal coordinates $u$ is obtained from $w$ via $u(e^{t},\theta)= e^{\frac{n-2s}{2}t}w(e^{t},\theta)$,
\item $v$ is deduced from $u$ by multiplying with the normal variable: $v=\theta_{n}^{\frac{1-2s}{2}}u$.
\end{itemize}
\end{rmk}

\begin{proof}[Proof of Proposition \ref{prop:KT1b}]
\emph{Step 1: Change of coordinates.} 
We carry out a change of coordinates, as this simplifies the handling of the duality formulation of the equation: We set $v=\theta_{n}^{\frac{1-2s}{2}}u$ and multiply (\ref{eq:op1}) with $\theta_n^{\frac{2s-1}{2}}$ from the left. In this formulation the conjugated version of equation (\ref{eq:op1}) turns into
\begin{equation}
\label{eq:sym}
\begin{split}
e^{\va(t)}\left(\dt^2 - \frac{(n-2s)^2}{4} + \theta_n^{\frac{2s-1}{2}}\nabla_{S^n}\cdot \theta_{n}^{1-2s}\nabla_{S^n}\theta_n^{\frac{2s-1}{2}} \right) e^{-\va(t)} v &= \theta_n^{\frac{2s-1}{2}}f,\\
\lim\limits_{\theta_{n}\rightarrow 0} \theta_{n}^{1-2s} \nu \cdot \nabla_{S^n}\theta_n^{\frac{2s-1}{2}}v &= h,
\end{split}
\end{equation}
where $\va=\tau \phi$.
In the new coordinates the desired Carleman inequality (\ref{eq:Carlflsymb}) then reads
\begin{equation*}
%\label{eq:Carlflsymv}
\begin{split}
&\tau^{-\frac{1}{2}} \left\| (\va''(t))^{\frac{1}{2}}\theta_{n}^{\frac{1-2s}{2}} \nabla_{S^n} \theta_n^{\frac{2s-1}{2}}v \right\|_{L^2(\R \times S^{n}_{+})} + \tau^{-\frac{1}{2}} \left\| (\va''(t))^{\frac{1}{2}}\dt v \right\|_{L^2(\R \times S^{n}_{+})}\\
& + \tau^{\frac{1}{2}} \left\|  (\va''(t))^{\frac{1}{2}} v \right\|_{L^{2}(\R \times S^{n}_{+})} 
+ \tau^{\frac{2s-1}{2}} \left\| (\va''(t))^{\frac{1}{2}} \lim\limits_{\theta_{n\rightarrow 0}} \theta_n^{\frac{2s-1}{2}} v \right\|_{L^2(\R \times \partial S^{n}_{+})}\\
\lesssim & \ \tau^{- \frac{1}{2}} \left\|  \theta_n^{\frac{2s-1}{2}} f\right\|_{L^2(\R \times S^{n}_{+})} 
+ \tau^{\frac{1-2s}{2}}  \left\| h \right\|_{L^{2}(\R \times \partial S^{n}_{+})} \mbox{ for } \tau \geq \tau_0 >0.
\end{split}
\end{equation*}
 
We test equation (\ref{eq:sym}) with eigenfunctions of the spherical operator $$\theta_n^{\frac{2s-1}{2}}\nabla_{S^n}\cdot \theta_{n}^{1-2s}\nabla_{S^n}\theta_n^{\frac{2s-1}{2}}$$ with vanishing generalized Neumann data.
Then equation (\ref{eq:sym}) turns into 
\begin{equation*}
e^{\va(t)}\left(\dt^2 - \lambda^2  - \frac{(n-2s)^2}{4}  \right)e^{- \va(t)} \E v = \E \theta_n^{\frac{2s-1}{2}}f + \Et h,
\end{equation*}
where we denote the projection of a function $v$ onto the eigenvector $v_{\lambda}$ by $\E v$ and its weighted boundary projection by $\Et v$. With a slight abuse of notation we will also use the symbol $\Et v$ for the scalar $\int\limits_{\p S^n_+}v \lim\limits_{\theta_n\rightarrow 0}\theta_n^{\frac{1-2s}{2}}v_{\lambda}d\mathcal{H}^{n-1}$.\\
The existence of a countable, diverging sequence of eigenvalues for the spatial part of the operator (\ref{eq:sym}) follows from the compactness of its inverse operator in an appropriate function space which is defined in the next step of the proof.  \\

\emph{Step 2: An Adapted Space.}
We define the analogues of the spaces $\dot {H}^1$ with the aid of our equation. Instead of the space $\dot{H}^{1}$, we use the modified space $\dot{H}^{1}_{\theta}$:
\begin{align*}
\dot{H}^{1}_{\theta} := \left\{ v  \Big|  \ \left\| \theta_{n}^{\frac{1-2s}{2}} \nabla_{S^n} \theta_{n}^{\frac{2s-1}{2}} v\right\|_{L^2(\R \times S_{+}^{n})} +  \left\| \dt v \right\|_{L^2(\R \times S_{+}^{n})} < \infty  \right\},
\end{align*}
and its semi-norm
\begin{align*}
&\left\| v\right\|_{\dot{H}^{1}_{\theta}} = \left\| v\right\|_{\dot{H}^{1}_{\theta,1}} + \left\| v \right\|_{\dot{H}^{1}_{\theta,2}} \mbox{ with } \\
&\left\| v \right\|_{\dot{H}^{1}_{\theta,1}}=  \left\| \theta_{n}^{\frac{1-2s}{2}} \nabla_{S^n} \theta_{n}^{\frac{2s-1}{2}}    v\right\|_{L^2(\R \times S_{+}^{n})},  \ \ 
\left\| v\right\|_{\dot{H}^{1}_{\theta,2}} =  \left\| \dt v \right\|_{L^2(\R \times S_{+}^{n})}. 
\end{align*}
We remark that intersected with $L^2_{loc}(\R\times S^n_+)$ (and augmented by the right boundary values), this space constitutes the natural setting for the weak formulation of (\ref{eq:sym}). Due to the compactness of the embedding $H^{1}(\theta_n^{1-2s}, S^n_+) \hookrightarrow L^2(\theta_n^{1-2s},S^{n}_+)$, the solution operator associated with the vanishing Neumann version of the spherical operator contained in (\ref{eq:sym}) is compact if we additionally impose a mean value condition on the spaces (more precisely, the mean value property should be phrased as $\int\limits_{S^n_+}\theta_n^{\frac{2s-1}{2}}v d\theta=0$). As a result, its inverse has the claimed sequence of diverging eigenvalues.\\

\textit{Step 3: A trace estimate.}
A key tool in obtaining the desired Carleman estimates consists of using the right trace estimates. We use the following interpolation inequality:

\begin{lem}
Let $s\in(0,1)$ and let $u:S^{n}_+\rightarrow \R$ be measurable. Then,
\begin{equation}
\label{eq:trint}
\left\| u \right\|_{L^2(S^{n-1})} \lesssim \tau^{1-s}\left\| \theta_n^{\frac{1-2s}{2}} u \right\|_{L^2(S^n_+)} + \tau^{-s}\left\| \theta_n^{\frac{1-2s}{2}} \nabla_{S^n} u \right\|_{L^2(S^{n}_+)}
\end{equation}
for $\tau >1$.
\end{lem}

\begin{proof}
By Herbst's inequality (or the Hardy-trace inequality) we have
\begin{align*}
\left\| |y'|^{-s} w_{1} \right\|_{L^2(\R^{n})} \lesssim \left\| y_{n+1}^{\frac{1-2s}{2}} \nabla w_{1} \right\|_{L^2(\R^{n+1}_{+})},
\end{align*}
with $y=(y',y_{n+1})$, $y'\in \R^{n}$ and $s\in(0,1)$.
Applied to functions supported in $B_1^+(0)$ this leads to
\begin{align*}
\left\| w_{1} \right\|_{L^2(B_{1}(0))} \lesssim \left\| y_{n+1}^{\frac{1-2s}{2}}\nabla w_{1} \right\|_{L^2(B_{1}^+(0))} + \left\| y_{n+1}^{\frac{1-2s}{2}} w_{1} \right\|_{L^2(B_{1}^+(0))}.
\end{align*}
Rescaling, i.e. setting $w_{1}(x)=w(\mu x)$, yields
\begin{align*}
\left\| w \right\|_{L^2(B_{\mu}(0))} \lesssim \mu^{s-1} \left\| y_{n+1}^{\frac{1-2s}{2}}w\right\|_{L^2(B_{\mu}^+(0))}
+ \mu^{s} \left\| y_{n+1}^{\frac{1-2s}{2}}\nabla w\right\|_{L^2(B_{\mu}^+(0))}.
\end{align*}
From this, it is possible to obtain the multiplicative form of the inequality:
\begin{align*}
\left\| w \right\|_{L^2(B_{\mu}(0))} \lesssim  \left\| y_{n+1}^{\frac{1-2s}{2}}w\right\|_{L^2(B_{\mu}^+(0))}^{s} \left\| y_{n+1}^{\frac{1-2s}{2}}\nabla w\right\|_{L^2(B_{\mu}^+(0))}^{1-s},
\end{align*}
which -- by scaling -- can be applied to arbitrary functions in $C^{\infty}_{0}(\R^{n+1}_+)$. As a consequence, we obtain the estimate
\begin{align*}
\left\| w \right\|_{L^2(B_{\mu}(0))} \lesssim \tau^{1-s} \left\| y_{n+1}^{\frac{1-2s}{2}}w\right\|_{L^2(B_{\mu}^+(0))}
+ \tau^{-s} \left\| y_{n+1}^{\frac{1-2s}{2}}\nabla w\right\|_{L^2(B_{\mu}^+(0))},
\end{align*}
for all $\mu \geq 0$.
It remains to localize this estimate to the sphere. This can be achieved by extending an arbitrary function $u:S^n_+ \rightarrow \R$ zero-homogeneously into a neighbourhood of $S^n_+$. Using a cut-off function $\eta$, we apply the previous estimate to $w=\eta \tilde{u}$, where $\tilde{u}$ corresponds to the (zero-homogeneous) extension of $u$.
This results in
\begin{align*}
\left\| u \right\|_{L^2(S^{n-1})} &\lesssim \left\| w \right\|_{L^2(\R^{n})} \\
&\lesssim \tau^{-s}\left\| y_{n+1}^{\frac{1-2s}{2}}\nabla \tilde{u} \right\|_{L^2(B_{2}^+\setminus B_{\frac{1}{2}}^+)} + (\tau^{-s} + \tau^{1-s})\left\| y_{n+1}^{\frac{1-2s}{2}} \tilde{u} \right\|_{L^2(B_{2}^+\setminus B_{\frac{1}{2}}^+)}\\
& \lesssim  \tau^{-s}\left\| y_{n+1}^{\frac{1-2s}{2}}\nabla_{S^n} u \right\|_{L^2(S^{n}_+)} + \tau^{1-s}\left\| y_{n+1}^{\frac{1-2s}{2}} u \right\|_{L^2(S^{n}_+)},
\end{align*}
for $\tau > 1$.
\end{proof}

\emph{Step 4: Conclusion.}
We conclude the argument with a commutator estimate: After the projection onto the eigenvectors, the operator becomes purely one-dimensional 
\begin{align*}
L \E v:&=\left(\dt^2+(\va'(t))^2-2\va'(t)\dt-\va''(t)-\lambda^2- \frac{(n-2s)^2}{4}\right)\E v\\
& = -\Et h + \E \theta_n^{\frac{2s-1}{2}} f.
\end{align*}
It decomposes into a symmetric and an antisymmetric part. Setting $\mu^2 := \lambda^2 + \frac{(n-2s)^2}{4}$ leads to
\begin{align*}
S&=\dt^2 + (\va')^2-\mu^2,\\
A&=-2\va'\dt - \va''.
\end{align*}
As its commutator reads
\begin{align*}
\int\limits_{\R} ([S,A]\E v,\E v)dt = \int\limits_{\R}\va''(\va')^2(\E v)^2 + \va''(\E  v')^2 - \va'''' (\E v)^2 dt, 
\end{align*}
we obtain the estimate
\begin{align*}
\left\| L \E v\right\|_{L^2(\R)}^2 \geq & \ \left\| S \E v \right\|_{L^2(\R)}^2+ \left\| A \E v \right\|_{L^2(\R)}^2+ \int\limits_{\R} ([S,A]\E v,\E v)dt\\
 \geq & \ \left\| (\dt^2+\va'^2 - \mu^2) \E v \right\|_{L^2(\R)}^2+ \left\| (2\va'\dt+\va'') \E v \right\|_{L^2(\R)}^2\\
&+  \int\limits_{\R}\va''(\va')^2(\E v)^2 + \va'' (\E v')^2 - \va'''' (\E v)^2 dt.
\end{align*}
By assumption $\va$ is a convex weight of the form $\va(t)=-\tau t + \tau \psi$ and $\psi''(t)=\frac{1}{10(1+t^2)}$. Hence, we observe that the first two commutator contributions are positive and it is possible to absorb the potentially negative $\va'''' (\E v)^2$ contribution of the commutator in the other positive contributions. Therefore, we obtain
\begin{align}
\label{eq:crit}
 \left\| (\va'')^{\frac{1}{2}}\va' \E v \right\|_{L^2(\R)} +  \left\| (\va'')^{\frac{1}{2}}\E v' \right\|_{L^2(\R)} \lesssim \left\| L \E v \right\|_{L^2(\R)}.
\end{align}
Moreover, we note that in the regimes $\lambda \geq 4\tau$ and $\lambda \leq \frac{\tau}{2}$ the symmetric part of the operator is elliptic (where the symbol of the operator is interpreted as a symbol in the $t$- \emph{and} $\tau$-variables), as by virtue of the definition of the weight function $|\va'|\in [\frac{3}{4}\tau,2\tau]$. Hence, by scaling we also obtain
\begin{align}
\label{eq:ell}
\lambda^2 \left\| \E v \right\|_{L^2(\R)} + \lambda \left\| \E v' \right\|_{L^2(\R)} \lesssim \left\| L \E v \right\|_{L^2(\R)} 
\end{align}
in these two elliptic regimes. By definition of the space $\dot{H}^{1}_{\theta}$, it holds
\begin{align*}
\left\| \E v \right\|_{\dot{H}^{1}_{\theta}(S^n_+)} = \lambda \left\| \E v \right\|_{L^2(S^n_+)}.
\end{align*}
Integrating the estimates (\ref{eq:crit}) and (\ref{eq:ell}) over $S^n_+$, shows
\begin{align*}
\left\| (\va'')^{\frac{1}{2}}(\va')\E v \right\|_{\dot{H}^{1}_{\theta}(S^n_+)L^2_t(\R)} \lesssim \left\| L \E v \right\|_{L^2(S^n_+\times \R)}. 
\end{align*}
Thus, these estimates yield the bulk contributions of the left hand side of the Carleman inequality. \\
We proceed by estimating the contributions on the right hand side of the Carleman inequality:
\begin{align*}
\left\| L\E v \right\|_{L^2(S^n_+ \times \R)} \leq  \left\| \E f \right\|_{L^2(S^n\times \R)} +  \left\| \Et h \right\|_{L^2(\p S^n\times \R)}. 
\end{align*}
In this context, it suffices to discuss the second term.
It can be estimated via the trace inequality:
\begin{align*}
 \left\| \Et h \right\|_{L^2(\p S^n\times \R)}
\leq \lambda^{1-s} \left\| \E h \right\|_{L^2(S^n \times \R)}. 
\end{align*}
In the low and critical frequency regimes, i.e. if $\lambda \leq 4\tau$, this can be estimated by $ \tau^{1-s} \left\| \E h \right\|_{L^2(S^n \times \R)}$. In the high frequency regime, we can also replace the $\lambda$ factor by $\tau$, as then the estimates become elliptic, e.g. the $L^2$ bulk estimate (with $f=0$) then reads
\begin{align*}
\left\| \E v \right\|_{L^2(S^n_+\times \R)} \lesssim \lambda^{-1-s}\left\| \E h\right\|_{L^2(S^n_+\times \R)} \lesssim \tau^{-1-s}\left\| \E h\right\|_{L^2(S^n_+\times \R)}.
\end{align*}
The other contributions can be treated analogously.
Combined with the previous considerations this implies the estimate
\begin{align*}
&\tau^{-\frac{1}{2}} \left\| (\va''(t))^{\frac{1}{2}}\theta_{n}^{\frac{1-2s}{2}} \nabla_{S^n} \theta_n^{\frac{2s-1}{2}}v \right\|_{L^2(\R \times S^{n}_{+})} + \tau^{-\frac{1}{2}} \left\| (\va''(t))^{\frac{1}{2}}\dt v \right\|_{L^2(\R \times S^{n}_{+})}\\
& + \tau^{\frac{1}{2}} \left\|  (\va''(t))^{\frac{1}{2}} v \right\|_{L^{2}(\R \times S^{n}_{+})} \\
\lesssim & \ \tau^{- \frac{1}{2}} \left\|  \theta_n^{\frac{2s-1}{2}} f\right\|_{L^2(\R \times S^{n}_{+})} 
+ \tau^{\frac{1-2s}{2}}  \left\| h \right\|_{L^{2}(\R \times \partial S^{n}_{+})}.
\end{align*}
Now, the estimate on the boundary contributions follows from the interpolation inequality (\ref{eq:trint}) applied to $u= \theta^{\frac{2s-1}{2}}_n v$, Fubini's theorem and the condition $|\va'| \in [\frac{3}{4}\tau,2\tau]$:
\begin{align*}
\tau^{s} \left\| (\va''(t))^{\frac{1}{2}} \lim\limits_{\theta_n \rightarrow 0} \theta_n^{\frac{2s-1}{2}}  v \right\|_{L^2(\R \times \p S^n_+)} &\lesssim
 \left\| (\va'')^{\frac{1}{2}}(\va') v \right\|_{L^2(\R^2)} +  \left\| (\va'')^{\frac{1}{2}} v' \right\|_{L^2(\R^2)}\\
& \lesssim \left\| L\E v \right\|_{L^2(\R)}. 
\end{align*}
This yields the full result.
\end{proof}

Before deducing further results from this, we pause for a few remarks:

\begin{rmk}[Spectral Gap]
\label{rmk:Carlsym}
The equations (\ref{eq:op1}) and (\ref{eq:sym}) contain the structure of the operator in its cleanest form: 
\begin{equation}
\label{eq:OP}
\begin{split}
\tilde{\D} &:= \dt^2 - \frac{(n-2s)^2}{4} + \theta^{- \frac{1-2s}{2}}_{n} \nabla_{S^n} \cdot \theta^{1-2s}_{n} \nabla_{S^n} \theta_{n}^{- \frac{1-2s}{2}}\\
& =: \dt^2  - \frac{(n-2s)^2}{4} - \tilde{\D}_{\theta}.
\end{split}
\end{equation}
The corresponding boundary values turn into
\begin{align*}
\lim\limits_{\theta_{n}\rightarrow 0} \theta_{n}^{1-2s} \nu \cdot \nabla_{S^{n}}\theta^{- \frac{1-2s}{2}}_{n}v = e^{2st}\lim\limits_{\theta_{n}\rightarrow 0}\theta^{-\frac{1-2s}{2}}_{n}V v.
\end{align*}
At first sight, one could hope that the eigenvalue expansion of the spherical operator leads to a situation comparable to that of Koch and Tataru \cite{KT1}. However, this is not clear. As the $\theta_n$-factors break the full rotation symmetry, the spectral gap of the spherical Laplacian need not be preserved.\\
%: The spherical operator can be expanded to yield
%where $C_{s} = \frac{(1-2s)(1+2s)}{4}$. For $s> \frac{1}{2}$ the operator is negative. Yet, there need not be a sufficiently strong spectral gap. \\
We note that \emph{in the case} of a spectral gap of constant strength, the Carleman inequality, (\ref{eq:Carlflsymb}), can be further improved by an estimate of the form
\begin{align*}
\dist(\va'(t), \spec(\tilde{\D}_{\theta})) \left\| \theta_{n}^{\frac{1-2s}{2}} u \right\|_{L^2(\R \times S^n_{+})} \\
\lesssim \ \tau^{-\frac{1}{2}}\left\| \theta^{\frac{2s-1}{2}}_n f \right\|_{L^2(\R \times S^n_+)}
&+\tau^{\frac{1-2s}{2}}\left\| h \right\|_{L^2(\R \times \p S^n_+)}.
\end{align*}
A similar remark holds for the gradient inequality. This type of estimates can, for example, be seen by constructing a parametrix for the operator
\begin{align*}
e^{-\va(t)}\left( \dt^2   - \frac{(n-2s)^2}{4} - \tilde{\D}_{\theta} \right) e^{\va(t)}
\end{align*}
on each eigenspace. Following Koch and Tataru \cite{KT1}, using $\va'(t)\leq 0$, $\va''(t)>0$ and setting 
\begin{equation*} 
\mu=  \sqrt{\frac{(n-2s)^2}{4} + \lambda^2}, 
\end{equation*}
the kernel of this parametrix reads
\begin{align*}
K_{\mu}(t,s)= e^{\va(t)-\va(s)} \left\{ 
\begin{array}{ll}
- \frac{1}{2} \mu^{-1} e^{- \mu|t-s|} & \mbox{ if } t> T(\mu),\\
\mu^{-1} \sinh(\mu(s-t)) & \mbox{ if } T(\mu)>t>s,\\
0 & \mbox{ if } T(\mu),s>t,
\end{array}
\right.
\end{align*}
on the eigenspace associated with the eigenvalue $\mu$. Here $T(\mu)$ is a solution of
\begin{align*}
\va'(t) = -\mu, 
\end{align*}
if $\mu$ is in the range of $\va'$ and else is defined as 
\begin{align*}
T(\mu)= \left\{
\begin{array}{ll}
- \infty &\mbox{ if } -\mu<\va',\\
+\infty & \mbox{ if } -\mu> \va'.
\end{array}
\right.
\end{align*}
Thus, using convexity, the kernel can be estimated by
\begin{align*}
|K_{\mu}(t,s)| \leq \tau^{-1}e^{- \dist(\va'(t),\mu)|t-s|}
\end{align*} 
in the critical regime in which $\va' \in [\frac{\tau}{2},4\tau]$.
Combined with Young's inequality and the estimates in the low and high frequency elliptic regimes, this implies the claimed $L^2$ bound.
We will use this for the one-dimensional fractional Laplacian, c.f. Section \ref{sec:1D}.
\end{rmk}

\begin{rmk}
\label{rmk:anti}
From the antisymmetric part of the operator we can obtain further $L^2$ bounds in combination with Poincar\'e's inequality. Using the same notation as in Remark \ref{rmk:not} and in the proof of Proposition \ref{prop:KT1b}, we assume that $w$ is supported in $\{\delta \leq |y| \leq R \}$ or in other words, $v$ is supported in $\{\ln(\delta) \leq t \leq \ln(R) \}\times S_{+}^n$. Then, for $0<c_{0}<c<C_{0}<\infty$ and $R\geq C_{0}\delta$, the antisymmetric operator can be estimated from below 
\begin{align*}
\left\| A v \right\|_{L^2((\ln(\delta), \ln(R))\times S^{n}_{+}) }^2  \geq & \ \left\| A v \right\|_{L^2((\ln(\delta), \ln(c \delta))\times S^{n}_{+}) }^2 \\
= & \ \tau^2 \left\| (2\dt \phi \dt+\dt^2 \phi) v \right\|_{L^2((\ln(\delta), \ln(c \delta))\times S^{n}_{+}) }^2\\
 \gtrsim & \ \tau^2 \left\| (\dt \phi ) \dt v \right\|_{L^2((\ln(\delta), \ln(c\delta))\times S^{n}_{+})}^2\\
& - \tau^2 \left\| (\dt^2 \phi)v \right\|_{L^2((\ln(\delta), \ln(c \delta))\times S^{n}_{+}) }^2.
\end{align*}
While considering the second quantity in this inequality as a controlled error contribution, we further estimate the first one. Using
\begin{align*}
\dt \phi \dt v = \dt (\dt \phi  v) - \dt^2 \phi  v
\end{align*}
as well as $\dt v|_{(e^{t},\theta)} = e^{-t}\dt g|_{(t,\theta)} $ with $g(t,\theta)=v(e^{t},\theta)$ in combination with Poincar\'e's inequality leads to:
\begin{align*}
\left\| A v \right\|_{L^2((\ln(\delta), \ln(R))\times S^{n}_{+}) }^2  \gtrsim & \ \tau^2 \delta^{-2}\left\| \dt \phi v \right\|_{L^2((\ln(\delta), \ln( c\delta))\times S^{n}_{+}) }^2\\
& - 2\tau^2 \left\| (\dt^2 \phi)v \right\|_{L^2((\ln(\delta), \ln(c \delta))\times S^{n}_{+}) }^2.
\end{align*}
Recalling the proof of the Carleman estimate (\ref{eq:Carlflsymb}), we observe that the right hand side of the inequality, in particular, bounds the antisymmetric part of the operator. In $v$-variables and using $\va= \tau \phi$, this amounts to the estimate
\begin{equation*}
%\label{eq:Carlflsymv}
\begin{split}
&\tau^{-\frac{1}{2}} \left\| (\va''(t))^{\frac{1}{2}}\theta_{n}^{\frac{1-2s}{2}} \nabla_{S^n} \theta_n^{\frac{2s-1}{2}}v \right\|_{L^2(\R \times S^{n}_{+})} + \tau^{-\frac{1}{2}} \left\| (\va''(t))^{\frac{1}{2}}\dt v \right\|_{L^2(\R \times S^{n}_{+})}\\
& + \tau^{\frac{1}{2}} \left\|  (\va''(t))^{\frac{1}{2}} v \right\|_{L^{2}(\R \times S^{n}_{+})} 
+ \tau^{-\frac{1}{2}}\left\| (2\va'(t)+ \va''(t))v \right\|_{L^2(\R \times S^{n}_{+})} \\
&+ \tau^{\frac{2s-1}{2}} \left\| (\va''(t))^{\frac{1}{2}} \lim\limits_{\theta_{n\rightarrow 0}} \theta_n^{\frac{2s-1}{2}} v \right\|_{L^2(\R \times \partial S^{n}_{+})}\\
\lesssim & \ \tau^{- \frac{1}{2}} \left\|  \theta_n^{\frac{2s-1}{2}} f\right\|_{L^2(\R \times S^{n}_{+})} 
+ \tau^{\frac{1-2s}{2}}  \left\| h \right\|_{L^{2}(\R \times \partial S^{n}_{+})} \mbox{ for } \tau \geq \tau_0 >0.
\end{split}
\end{equation*}
Hence, as the error term can be absorbed in the Carleman inequality, the estimate from above corresponds to 
\begin{align*}
&\tau^2 \delta^{-2} \left\| e^{\tau \phi} y_{n+1}^{\frac{1-2s}{2}} w \right\|_{L^2(B_{c \delta}\setminus B_{ \delta}) }^2 \lesssim  \ \left\|e^{\tau \phi}  y_{n+1}^{\frac{2s-1}{2}} |y| \nabla\cdot y^{1-2s}_{n+1}\nabla w \right\|_{L^2(B_{R}\setminus B_{ \delta}) }^2
\end{align*}
in Cartesian coordinates (if $h=0$).
\end{rmk}

\subsection{Consequences of the Carleman Estimate (\ref{eq:Carlflsymb})}
\label{sec:cor}
From the previous estimates we obtain a unique continuation result in the case of infinite order vanishing in both the tangential and normal directions.

\begin{cor}[SUCP I]
\label{cor:schp}
Let $s\in(0,1)$ and let $w:\R^n \rightarrow \R$, $w\in H^{s}$, be a solution of
\begin{align*}
(-\D)^{s} w = Vw,
\end{align*}
with $V= V_{1} + V_{2}$,
\begin{align*}
V_{1}(y) = |y|^{-2s}h\left( \frac{y}{|y|} \right), \ \ h\in L^{\infty} ,\ \ |V_{2}(y)| \leq c|y|^{-2s+\epsilon}. %, \ \ \left\| h\right\|_{L^{\infty}}\ll 1 .
\end{align*}
For $s < \frac{1}{2}$, we additionally require that one of the following assumptions is satisfied:
\begin{itemize}
\item the potential $V_2$ satisfies
$
V_{2}\in C^{1}(\R^n\setminus \{0\}) \mbox{ and }
|y\cdot \nabla V_{2}| \lesssim c|y|^{-2s+\epsilon},
$
\item $s\in [\frac{1}{4}, \frac{1}{2})$ and $V_1 \equiv 0$.
\end{itemize}

Let $\tilde{w}$ denote the Caffarelli extension of $w$.
If  $\tilde{w}$ vanishes of infinite order at $0$ in both the tangential and normal directions, 
then
\begin{align*}
w \equiv 0.
\end{align*}
\end{cor}

\begin{rmk}
\label{rmk:vani}
Before presenting the proof, we deduce an estimate for (critically) weighted boundary terms, involving e.g. $V(y) \sim |y|^{-2s}$. Due to the infinite order of vanishing of $w$ along the boundary, for any $\epsilon >0 $ and any $m\in \N$, there exists a radius $\bar{r} = \bar{r}(m,\epsilon) >0$ such that for all $0<r\leq \bar{r}$
\begin{align*}
\int\limits_{B_{2r}\setminus B_{r}} |V|w^2 dy \lesssim r^{-2s}\int\limits_{B_{2r}\setminus B_{r}}w^2 dy \leq \epsilon r^{m}.
\end{align*}
As a result,
\begin{align*}
\int\limits_{B_{2r}} |V|w^2 dy = & \ \sum\limits_{j\in \N} \int\limits_{B_{2^{-j}r}\setminus B_{2^{-j-1}r}} |V|w^2 dy  \\
\lesssim & \ \epsilon r^m \sum\limits_{j\in \N} 2^{-jm}\\
\lesssim  &\ \epsilon r^m.
\end{align*}
Hence, the infinite rate of vanishing of $u$ on the boundary also implies that (singularly) weighted boundary integrals have an infinite rate of vanishing.
\end{rmk}

We present the proof for subcritically scaling potentials in the case $s\geq \frac{1}{4}$ first. Then we indicate how to modify the previous arguments for $0<s<\frac{1}{4}$ and in the case of scale-invariant potentials. 

\begin{proof}[Proof in the Case of Subcritical Potentials and $s\geq \frac{1}{4}$]
\textit{Step 1: Interpolation.} For $w \in C^{\infty}_{0}(\overline{Q_{\epsilon}^{+}})$ and $0\leq \epsilon \leq \frac{1}{2}$ the following interpolation inequality holds true:
\begin{align*}
\frac{1}{\epsilon^2} \int\limits_{Q_{\epsilon}^+}y^{1-2s}_{n+1}|\nabla w|^2 dy \lesssim
\frac{C(\mu)}{\epsilon^4} \int\limits_{Q_{\epsilon}^+} y_{n+1}^{1-2s}|w|^2 dy +
\mu^2 \int\limits_{Q_{\epsilon}^+} y_{n+1}^{2s-1}|\nabla \cdot y^{1-2s}_{n+1}\nabla w|^2 dy\\
 - \frac{1}{\epsilon^2} \int\limits_{Q_{\epsilon}^+\cap \{y_{n+1}=0\}} w y_{n+1}^{1-2s}\p_{n+1}w dy',
\end{align*}
where $Q_{\epsilon}^+=[-\epsilon, \epsilon]^{n}\times [0,\epsilon]$.
This estimate will be employed in deriving the infinite order of vanishing of the gradient from the infinite order of vanishing of $\frac{1}{\epsilon^4} \int\limits_{B_{\epsilon}} y_{n+1}^{1-2s}w^2 dy$ for (almost) solutions.\\
The inequality is a result of integration by parts and the support condition on $w$. In fact, we have
\begin{align*}
\int\limits_{Q_{1}^+} y_{n+1}^{1-2s}|\nabla' w|^2 dy = & \ 
-\int\limits_{0}^{1} y_{n+1}^{1-2s} \int\limits_{[-1,1]^{n}} w \D' w(\cdot,y_{n+1}) dy' dy_{n+1} .
\end{align*} 
Moreover,
\begin{align*}
\int\limits_{Q_{1}^+} y_{n+1}^{1-2s}| \p_{n+1} w|^2 dy = & \ 
-\int\limits_{[-1,1]^{n}} \int\limits_{0}^{1} w(\p_{n+1} y_{n+1}^{1-2s} \p_{n+1} w)  dy_{n+1} dy' \\
& \ - \int\limits_{[-1,1]^n \times \{0\}} w y_{n+1}^{1-2s}\p_{n+1}w dy'. 
\end{align*}
Combining these two estimates yields
\begin{align*}
\int\limits_{Q_{1}^+}y^{1-2s}_{n+1}|\nabla w|^2 dy 
\leq & \ \int\limits_{Q_{1}^+}| w \nabla \cdot y_{n+1}^{1-2s}\nabla w| dy
+ \int\limits_{[-1,1]^n \times \{0\}} w y_{n+1}^{1-2s}\p_{n+1}w dy\\
\leq &  \ \mu^2 \int\limits_{Q_{1}^+}  y_{n+1}^{2s-1}|\nabla \cdot y_{n+1}^{1-2s}\nabla w|^2 dy + 
C(\mu) \int\limits_{Q_{1}^+}  y_{n+1}^{1-2s}|w|^2 dy \\
& \ - \int\limits_{[-1,1]^n \times \{0\}} w y_{n+1}^{1-2s}\p_{n+1}w dy.
\end{align*}
The claimed inequality now follows from scaling.\\

\textit{Step 2: Cut-off Errors.} Denoting the Caffarelli extension of $w$ by $\tilde{w}$, we consider $\bar{w} = \tilde{w}\eta_{\delta,r}$ where $\e$ is a radial cut-off function which equals one on an annulus with radii approximately determined by $ \delta $ and $r$ where $0< \delta \ll r <1$.
Thus, $\bar{w}$ satisfies
\begin{equation}
\begin{split}
\label{eq:errorf}
\nabla \cdot y^{1-2s}_{n+1} \nabla \bar{w} = & \ y_{n+1}^{1-2s}\e'' \tilde{w} + y^{1-2s}_{n+1}  \e' \frac{y}{|y|}  \cdot \nabla \tilde{w} \\
& + (n+1-2s)y_{n+1}^{1-2s} \frac{1}{|y|}\e' \tilde{w} \mbox{ in } \R^{n+1}, \\
-\lim\limits_{y_{n+1}\rightarrow 0} y_{n+1}^{1-2s} \p_{n+1} \bar{w} &= V \bar{w} \mbox{ on } \R^{n}.
\end{split}
\end{equation}
Due to the cut-off, it is an admissible function in the Carleman inequality of Proposition \ref{prop:KT1b}. Inserting it into the Carleman inequality, we notice that we may pass to the limit $\delta \rightarrow 0$:
This follows from step 1 (in which $\mu$ is chosen sufficiently small) and the infinite order of vanishing of $\bar{w}$. Hence, the only remaining cut-off is at the scale $r>0$.\\

\emph{Step 3: Conclusion for Potentials with Subcritical Scaling.} % comment on critically scaling potentials, adapt to Cartesian variables, \va'' ausschreiben!
We consider the different contributions of the Carleman inequality:
\begin{equation*}
\begin{split}
& \left\| e^{\tau \phi}\left(1+\ln(|y|)^2\right)^{-\frac{1}{2}}y_{n+1}^{\frac{1-2s}{2}} \nabla w \right\|_{L^2(\R^{n+1}_{+} )}  \\ 
&+ \tau \left\| e^{\tau \phi}\left(1+\ln(|y|)^2\right)^{-\frac{1}{2}}y_{n+1}^{\frac{1-2s}{2}}|y|^{-1} w \right\|_{L^{2}(\R^{n+1}_{+})}\\
&+ \tau^{s} \left\| e^{\tau \phi}\left(1+\ln(|y|)^2\right)^{-\frac{1}{2}}|y|^{-s} w \right\|_{L^2(\R^{n})}\\
\lesssim & \ \tau^{- \frac{1}{2}} \left\|e^{\tau \phi} |y| y_{n+1}^{\frac{2s-1}{2}} f\right\|_{L^2(\R^{n+1}_{+})} 
\ + \tau^{\frac{1-2s}{2}}  \left\| e^{\tau \phi} |y|^{s} h \right\|_{L^{2}(\R^{n})}.
\end{split}
\end{equation*}
As all the right hand side terms of (\ref{eq:errorf}) involve derivatives of $\eta_{0,r}$ (which are, in particular, only active at scales $r>0$), they can be treated as controlled perturbations. Thus, it remains to investigate the boundary contributions. We recall $|V(y)| \leq |y|^{-2s+\epsilon}$.
This leads to a boundary contribution of
\begin{align*}
\tau^{s} \left\| (1+\ln(|y|)^2)^{-\frac{1}{2}} |y|^{-s} w \right\|_{L^2}
\end{align*}
on the left hand side of the Carleman inequality, and a contribution of the form
\begin{align}
\label{eq:rhsC} 
\tau^{\frac{1-2s}{2}} \left\|  |y|^{s} h \right\|_{L^2} \lesssim \tau^{\frac{1-2s}{2}} \left\| |y|^{-s+\epsilon} w\right\|_{L^2},
\end{align}
on the right hand side of the Carleman estimate. We note that in the case $s\geq \frac{1}{4}$ the $\tau$ contributions on the right hand side of the Carleman estimate are smaller or equal to the $\tau$ contributions on the left hand side. Thus, a strategy in which the dangerous terms of the right hand side are absorbed in the left hand side of the Carleman inequality is possible.
By virtue of the choice of the cut-off $\eta_{0,r}$, it suffices to consider $|y|<r$. 
Due to the subcriticality of $V$ and as the loss on the left hand side of the Carleman inequality is only logarithmic, the term on the right hand side of (\ref{eq:rhsC}) can be absorbed in the left hand side of the Carleman inequality. In the limit $\tau \rightarrow \infty$ this yields the desired result for (rough) subcritical potentials.
\end{proof}

In the sequel, we comment on the proof of Corollary \ref{cor:schp} in the case $s\in(0,\frac{1}{4})$ and in the setting involving scale-invariant potentials. For this, we argue via slightly different methods in obtaining the crucial Carleman estimates: In contrast to the previous arguments we do not carry out a decomposition into the spherical eigenvalues but work with the full operator.

\begin{proof}[Proof for $s\in (0,\frac{1}{2})$ and for Scaling-Critical Potentials]
\textit{Step 1: Conjugation and \\bulk contributions.}
We carry out the Carleman argument without projecting onto eigenvalues of the spherical operator. We start with the operator in conformal coordinates
\begin{align*}
\dt^2 - \frac{(n-2s)^2}{4} + \theta^{- \frac{1-2s}{2}}_{n}\nabla_{S^{n}}\cdot \theta_{n}^{1-2s} \nabla_{S^{n}}\theta^{- \frac{1-2s}{2}}_{n}.
\end{align*}
Conjugation with an only $t$-dependent weight $\phi$, leads to the following symmetric and antisymmetric parts of the operator:
\begin{align*}
S & = \dt^2 + \tau^2(\dt \phi)^2 - \frac{(n-2s)^2}{4} + \theta^{-\frac{1-2s}{2}}_{n}\nabla_{S^{n}}\cdot \theta_{n}^{1-2s}\nabla_{S^{n}}\theta^{- \frac{1-2s}{2}}_{n},\\
A & = -2\tau (\dt \phi)\dt - \tau \dt^2 \phi.
\end{align*}
If $\phi$ is sufficiently pseudoconvex this yields positive commutator terms. Furthermore, weighted gradient estimates can be obtained:
\begin{align*}
& ((\dt^2 \phi) \dt v, \dt v) + ((\dt^2 \phi)\theta^{1-2s}_{n} \nabla_{S^{n}}\theta_{n}^{-\frac{1-2s}{2}}v,\nabla_{S^{n}}\theta_{n}^{-\frac{1-2s}{2}}v)\\
=& \ -(Sv,(\dt^2 \phi) v) + \int\limits_{\p S^{n}_{+}\times \R} (\dt^2 \phi)\theta^{1-2s}_{n} (\nu \cdot \nabla_{S^{n}}\theta^{- \frac{1-2s}{2}}_{n}v) \theta^{- \frac{1-2s}{2}}_{n}v d\theta dt + ((\dt^4 \phi) v, v)\\
 \leq & \ \frac{1}{2 \tau^2} \left\| Sv \right\|_{L^2}^2 + \frac{1}{2}\tau^2 \left\| (\dt^2 \phi) v\right\|_{L^2}^2
+ \tau^2 \left\| (\dt^2 \phi)^{\frac{1}{2}}\dt \phi v \right\|_{L^2}^2 - \frac{(n-2)^2}{4} \left\| (\dt^2 \phi)^{\frac{1}{2}} v \right\|_{L^2}^2\\
&+\int\limits_{\p S^{n}_{+}\times \R} (\dt^2 \phi) \theta^{1-2s}_{n} (\nu \cdot \nabla_{S^{n}}\theta^{- \frac{1-2s}{2}}_{n}v) \theta^{- \frac{1-2s}{2}}_{n}v d\theta dt
 + ( (\dt^4 \phi) v, v)_{L^2},
\end{align*}
where $\nu = (0,...,0,-1)$ denotes the outer unit normal.
For sufficiently pseudoconvex $\phi$ the right hand side can be controlled by the commutator contributions of the Carleman estimate. In fact, this can even be strengthened by noticing that the right hand side remains controlled if it is multiplied by a factor of $c\tau$, with $c$ sufficiently small; for example $c \sim \frac{1}{2}$ would work. 
The boundary integral can be evaluated to yield
\begin{align*}
\int\limits_{\p S^{n}_{+}\times \R}  (\dt^2 \phi) (\theta^{1-2s}_{n}\nu \cdot \nabla_{S^{n}}\theta^{- \frac{1-2s}{2}}_{n}v) \theta^{- \frac{1-2s}{2}}_{n}v d\theta dt = \int\limits_{\p S^{n}_{+}\times \R} \theta_{n}^{-(1-2s)} (\dt^2 \phi) e^{2st} Vv^2 d\theta dt.
\end{align*}
The remaining boundary integral which originates from the commutator calculation is given by
\begin{align*}
&4\tau \int\limits_{\p S^{n}_{+}\times \R} \theta^{1-2s}_{n} (\nu \cdot \nabla_{S^{n}}\theta^{- \frac{1-2s}{2}}_{n}v) (\dt \phi) \theta^{- \frac{1-2s}{2}}_{n}\dt v d\theta dt \\
&+ 2\tau \int\limits_{\p S^{n}_{+}\times \R} \theta^{1-2s}_{n} (\nu \cdot \nabla_{S^{n}}\theta^{- \frac{1-2s}{2}}_{n}v) (\dt^2 \phi) \theta^{- \frac{1-2s}{2}}_{n} v d\theta dt \\
= & \ 4\tau \int\limits_{\p S^{n}_{+}\times \R} (\dt \phi) \theta_{n}^{-(1-2s)} e^{2st}V v \dt v d\theta dt
+ \ 2\tau \int\limits_{\p S^{n}_{+}\times \R} (\dt^2 \phi) \theta_{n}^{-(1-2s)}  e^{2st}V v^2 d\theta dt.
\end{align*}

Rewritten in terms of $u = \theta_{n}^{-\frac{1-2s}{2}} v$ the Carleman estimate reads
\begin{equation*}
%\label{eq:Carl}
\begin{split}
&c\tau \left\| (\dt^2 \phi)^{\frac{1}{2}} \theta^{\frac{1-2s}{2}}_{n} \dt u \right\|_{L^2}^2
+ c\tau \left\| (\dt^2 \phi)^{\frac{1}{2}} \theta^{\frac{1-2s}{2}}_{n} \nabla_{S^{n}} u \right\|_{L^2}^2
+ c\tau^3  \left\| \theta^{\frac{1-2s}{2}}_{n}(\dt^2 \phi)^{\frac{1}{2}}(\dt \phi) u \right\|_{L^2}^2\\
&+ \left\| S (\theta^{\frac{1-2s}{2}}_{n} u) \right\|_{L^2}^2
+ \tau^{-1} \left\| (\dt^2 + \theta^{-\frac{1-2s}{2}}_{n} \nabla_{S^{n}}\cdot \theta_{n}^{1-2s} \nabla_{S^n}  \theta^{- \frac{1-2s}{2}}_{n}) \theta^{ \frac{1-2s}{2}}_{n} u \right\|_{L^2}^2\\
& + 4\tau \int\limits_{\p S^{n}_{+}\times \R} (\dt \phi) e^{2st} V u \dt u d\theta dt
+ 2\tau \int\limits_{\p S^{n}_{+}\times \R} (\dt^2 \phi) e^{2st} V u \dt u d\theta dt\\
&+ c\tau \int\limits_{\p S^{n}_{+}\times \R}(\dt^2 \phi) e^{2st} V u \dt u d\theta dt\\
\leq & \ \left\| L_{\phi} u \right\|_{L^2}^2,
\end{split}
\end{equation*}
where
\begin{align*}
L_{\phi} = & \ \theta^{ \frac{1-2s}{2}}_{n}(\dt^2 + \tau^2 (\dt \phi)^2 - \frac{(n-2s)^2}{4} - 2\tau(\dt \phi)\dt - \tau \dt^2 \phi) \\
&+ \theta^{- \frac{1-2s}{2}}_{n} \nabla_{S^{n}}\cdot \theta_{n}^{1-2s} \nabla_{S^n}.
\end{align*}
Inserting the changes we made, i.e. $w = e^{\frac{n-2s}{2}t} u$, and recalling the changes in the volume element, yields a Carleman inequality which, up to the boundary contributions, is comparable to (\ref{eq:Carlflsymb}). \\

\textit{Step 2: Boundary Contributions under Differentiability Assumptions.}
In order to obtain a unique continuation statement as in Corollary \ref{cor:schp}, it remains to deal with the boundary contributions. We first present the argument under the differentiability assumption 
\begin{align*}
V_{2}\in C^{1}(\R^n\setminus\{0\}), \ |y\cdot \nabla V_{2}| \leq c|y|^{-2s+\epsilon},
\end{align*}
independently of the value of $s\in(0,1)$.
In order to estimate the unsigned boundary contributions, we consider the respective expressions in $u$-coordinates. Starting with the scaling-critical Hardy potentials, we have to bound
\begin{align*}
& 4\tau \int\limits_{\p S^{n}_{+}\times \R} (\dt \phi) e^{2st} V_{1} u \dt u d\theta dt
+ 2\tau \int\limits_{\p S^{n}_{+}\times \R} (\dt^2 \phi) e^{2st} V_{1} u^2 d\theta dt\\
&+ c\tau \int\limits_{\p S^{n}_{+}\times \R}(\dt^2 \phi) e^{2st} V_{1} u^2 d\theta dt,
\end{align*}
i.e. we have to control the boundary integrals involving the potential $V_{1}=e^{-2st}h(\theta)$. By an integration by parts in $t$, we obtain that most contributions drop out. Indeed, the only non-vanishing term is given by
\begin{align*}
%\label{eq:V1}
c\tau \int\limits_{\p S^{n}_{+}\times \R}\dt^2 \phi h(\theta) u^2 d\theta dt.
\end{align*}
This can be controlled via the interpolation inequality (\ref{eq:trint}): 
\begin{align*}
\tau \int\limits_{\p S^{n}_{+}\times \R}(\dt^2 \phi) h( \theta) u^2 d\theta dt \lesssim & \
\tau^{1-s} \left\| (\dt^2 \phi)^{\frac{1}{2}} \theta^{\frac{1-2s}{2}}_{n} \nabla_{S^{n}} u \right\|_{L^2}^2 + \tau^{2-s} \left\|  \theta_{n}^{\frac{1-2s}{2}} (\dt^2 \phi)^{\frac{1}{2}} u \right\|_{L^2}^2.
\end{align*}
All the remaining boundary contributions involve the potential $V_{2}$ which has subcritical growth at zero. Due to the form of $\phi$, it suffices to deduce control of the term
\begin{align*}
4\tau \int\limits_{\p S^{n}_{+}\times \R} \dt \phi e^{2st} V_{2} u \dt u d\theta dt.
\end{align*}
Integrating by parts in $t$, using the subcriticality of $V_{2}$ and the properties of $\phi$, it suffices to bound
\begin{align*}
C\tau \int\limits_{\p S^{n}_{+}\times \R} e^{\epsilon t}  u^2 d\theta dt.
\end{align*}
As the condition on the support of $u$ implies that $t< 0$, this can once more be achieved via the interpolation inequality (\ref{eq:trint}). \\

\textit{Step 3: Scaling-Critical Potentials for $s\geq \frac{1}{2}$.} Last but not least, we indicate how to prove the desired Carleman estimate in cases involving scaling-critical potentials without the differentiability assumptions from the previous step. While the scaling-critical potential, $V_{1}$, can be treated as the potentials in step 2, the subcritical part of the potential, $V_{2}$, cannot be differentiated. Thus, a direct estimate of this boundary term is needed. This is achieved via interpolation and regularity estimates for the operators. We only present the argument for the most critical boundary contribution which (after localization to a small radius $0<r\ll 1$) in Cartesian coordinates reads:
\begin{align*}
\tau \int\limits_{B_{r}^{+} \cap \{y_{n+1}=0\}}|V_2||w(y\cdot \nabla w)|dy.
\end{align*}
We estimate
\begin{equation}
\label{eq:boundary}
\begin{split}
\tau \int\limits_{B_{r}^+ \cap\{y_{n+1}=0\}}|V_2||w(y\cdot \nabla w)|dy \lesssim &\ \tau^2 \int\limits_{B_{r}^+\cap \{y_{n+1}=0\}}|y|^{-2s+\epsilon}w^2 dy \\ & \ + \int\limits_{B_{r}^+\cap \{y_{n+1}=0\}}|y|^{2-2s+\epsilon}|\nabla w|^2 dy.
\end{split}
\end{equation}
The first term can directly be interpolated between controlled quantities:
\begin{align*}
\tau^2 \int\limits_{B_{r}^+\cap \{y_{n+1}=0\}}|y|^{-2s+\epsilon}w^2 dy \lesssim \tau \left\| |y|^{\frac{\epsilon}{2}} y_{n+1}^{\frac{1-2s}{2}} \nabla w \right\|_{L^2(B_{r}^+)} + \tau^{3}\left\| |y|^{\frac{\epsilon}{2}-1} y_{n+1}^{\frac{1-2s}{2}} w \right\|_{L^2(B_{r}^+)}. 
\end{align*}
Here we have used $s\in [\frac{1}{2},1)$. The second quantity in (\ref{eq:boundary}) has to be controlled using elliptic estimates. Due to $L^2$ estimates for the respective degenerate elliptic Neumann boundary value problem (which one can for example deduce by carrying out a tangential Fourier transform), we have
\begin{align*}
&  \int\limits_{B_{r}^{+}\cap\{y_{n+1}=0\}} |y|^{2-2s +\epsilon}|\nabla w|^2 dy \\
\lesssim & \ \tau^{-1}\left\| |y|^{1+\frac{\epsilon}{2}}(y_{n+1}^{\frac{2s-1}{2}} \nabla \cdot y_{n+1}^{1-2s} \nabla  + \tau^2 |\nabla \phi|^2 y^{\frac{1-2s}{2}}_{n+1} )w  \right\|_{L^2(B_{2r}^{+})}^2\\
& + \tau^{2}\left\||y|^{-1 + \frac{\epsilon}{2}} y^{\frac{1-2s}{2}}_{n+1} w   \right\|_{L^2(B_{2r}^{+})}^2 
+ \tau \left\| |y|^{\frac{\epsilon}{2}}y_{n+1}^{\frac{1-2s}{2}}\nabla w \right\|_{L^2(B_{2r}^+)}^2 \\
&+ \tau^{2-4s}\left\| |y|^{s+\frac{\epsilon}{2}} V w \right\|_{L^2(B_{2r}^+ \cap \{y_{n+1}=0\})}.
\end{align*}
As all the right hand side terms are controlled by the bulk terms of the Carleman inequality, we can also control perturbations of critically scaling potentials without imposing differentiability constraints on the perturbation.
\end{proof}

\section[Doubling and the WUCP]{Doubling Estimates and Reduction to the Weak Unique Continuation Property}
\label{sec:dewucp}
\subsection{Doubling Inequalities}
\label{sec:doubl}
In this section we deduce a doubling inequality which plays a decisive role in the compactness argument reducing the strong to the weak unique continuation property.
We have

\begin{prop}
\label{prop:doubl}
Let $s\in(0,1)$ and let $w: \R^{n+1} \rightarrow \R$, $w\in H^{1}_{loc}(y_{n+1}^{1-2s}dy,\R^{n+1})\cap H^{2}_{loc}(y_{n+1}^{1-2s}dy,\R^{n+1})$, be a solution of
\begin{align*}
\nabla \cdot y_{n+1}^{1-2s} \nabla  w & = 0 \mbox{ in } \R^{n+1}_{+}, \\
-\lim\limits_{y_{n+1}\rightarrow 0} y^{1-2s}_{n+1} \p_{n+1} w &= Vw \mbox{ on } \R^n,
\end{align*}
with $V= V_{1} + V_{2}$,
\begin{align*}
V_{1}(y) = |y|^{-2s}h\left( \frac{y}{|y|} \right), \ \ h\in L^{\infty}, \ \ |V_{2}(y)| \leq c|y|^{-2s+\epsilon}. %, \ \ \left\| h \right\|_{L^{\infty}}\ll 1.
\end{align*}
For $s < \frac{1}{2}$, we additionally require that one of the following assumptions is satisfied:
\begin{itemize}
\item the potential $V_2$ satisfies
$
V_{2}\in C^{1}(\R^n\setminus \{0\}) \mbox{ and }
|y\cdot \nabla V_{2}| \lesssim c|y|^{-2s+\epsilon},
$
\item $s\in [\frac{1}{4}, \frac{1}{2})$ and $V_1 \equiv 0$.
\end{itemize}
Then the doubling property holds, i.e. there exists a constant $C>0$ and a constant $R$ 
such that for all $0<r<R$ we have
\begin{align*}
\left\| y^{\frac{1-2s}{2}}_{n+1} w \right\|_{L^2(B_{2r}^+(0))} \leq C \left\| y^{\frac{1-2s}{2}}_{n+1} w \right\|_{L^2(B_{r}^+(0))}.
\end{align*}
\end{prop}

Before commencing with the proof of the doubling property, a few remarks are in order:
\begin{rmk}
\label{rmk:doubl1}
\begin{itemize}
\item We note that the doubling property can be shown for any $R>0$. However, in order to obtain a uniform dependence of $C$ on $r$, this parameter has to be fixed. 
\item We point out that the constant $C>0$ depends on the function $w$. 
\item The doubling property is neither restricted to balls centered at the origin nor to balls centered at the boundary of $\R^{n+1}_+$. Under the conditions of Proposition \ref{prop:doubl} the conclusion can be formulated as the existence of a constant $C>0$ and a constant $R$ such that for all $0<r<R$ and for all $y_0 \in B_{R}(z)$, $z\in \R^{n+1}_+$, we have
\begin{align*}
\left\| y^{\frac{1-2s}{2}}_{n+1} w \right\|_{L^2(B_{2r}(y_0)\cap \R^{n+1}_+)} \leq C \left\| y^{\frac{1-2s}{2}}_{n+1} w \right\|_{L^2(B_{r}(y_0)\cap \R^{n+1}_+)}.
\end{align*}
In this case $C=C(R,z,w)$. We comment on the proof of this more general statement after the proof of Proposition \ref{prop:doubl}, c.f. Remark \ref{rmk:doubl}.
\end{itemize}
\end{rmk}

\begin{proof} % direct comment on boundary values: absorb, via Cacciopolli, adapt weight depending on $\va''$
Without loss of generality, we restrict our attention to sufficiently small radii and to balls centered at the origin. Via a covering argument, it is possible to recover the statement for larger balls, c.f. Remark \ref{rmk:doubl}. In order to bound the gradient contributions which will arise in the application of the Carleman inequality (\ref{eq:Carlflsymb}), we recall the following elliptic gradient/ Cacciopolli estimate:
Let $\psi$ be a cut-off function supported in an annulus given by $0<\frac{r_0}{2} \leq|y| \leq  2r_{1}<\infty$, which we will also denote by $(\frac{r_0}{2},  2r_{1})$ in the sequel. Then,
\begin{equation}
\begin{split}
\label{eq:ellreg}
\left\| y_{n+1}^{\frac{1-2s}{2}} \nabla (w\psi) \right\|_{L^2(\frac{r_{0}}{2},2r_{1})}^2
\lesssim & \ r_{0}^{-2}\left\| y_{n+1}^{\frac{1-2s}{2}}  w \right\|_{L^2(r_{0}/2,2r_{1})}^2 \\& + \int\limits_{ (\frac{r_{0}}{2},2r_{1}) \cap \{y_{n+1}=0\}} \psi w \lim\limits_{y_{n+1}\rightarrow 0} y^{1-2s}_{n+1} \p_{n+1}( \psi w) dy,
\end{split}
\end{equation}
with $0<r_{0}<r_{1}<\infty$. If the boundary conditions are of the generalized Neumann type as in our assumptions, it becomes possible to absorb these into the left-hand side bulk gradient term, if they are sufficiently small, i.e. if $V$ is either subcritical or if it is a small scaling-critical potential. In the case of large scaling-critical potentials it is still possible to absorb these contributions, if the vanishing rate in the tangential direction is higher than in the normal direction. By virtue of Corollary \ref{cor:schp} it is always possible to reduce to this situation.\\
Keeping this in mind, we prepare for the application of the Carleman inequality from Proposition \ref{prop:KT1b}: Let $\eta$ be a radial cut-off function, which is equal to one on the annulus $|y|\in (\delta, \tilde{R}/2)$ and vanishes outside of the annulus $|y|\in (\delta/2, \tilde{R})$. 
Inserting $\eta w$ into the Carleman estimate (in combination with Remark \ref{rmk:anti}), using the elliptic estimate as well as the explicit form of the boundary contribution, we obtain
\begin{align*}
&\delta^{-2} \tau \left\| e^{\tau \phi} y^{\frac{1-2s}{2}}_{n+1} w \right\|_{L^2(\delta, 3\delta)}^2 \\
&+ \tau^2 \tilde{R}^{-2} \left\| e^{\tau \phi} (1+\ln(|y|)^{2})^{-\frac{1}{2}}y^{\frac{1-2s}{2}}_{n+1} w \right\|_{L^2(\tilde{R}/8, \tilde{R}/4)}^2\\
\lesssim & \ \delta^{-2} \left\| e^{\tau \phi} y^{\frac{1-2s}{2}}_{n+1} w \right\|_{L^2(\delta /2, \frac{3\delta}{2})}^2 + \tilde{R}^{-2}\left\| e^{\tau \phi} y^{\frac{1-2s}{2}}_{n+1} w \right\|_{L^2(\tilde{R}/2,2\tilde{R})}^2.
\end{align*}
Here the boundary contributions were absorbed into the bulk contributions in the way indicated above.
%Here the boundary contributions from the Carleman estimate and the fact that $0<\delta< \tilde{R} < \infty$ allow to absorb the boundary contributions from the elliptic estimates into the left hand side of the Carleman inequality if the parameter $\tau\geq \tau_0>0$ is chosen sufficiently large. % no trouble/ case distinction with scaling critical potentials as these are included in the operator as indicated in the proof of corollary 1
Setting $\tilde{R}\sim 1$, we estimate further 
\begin{align*}
e^{\tau \phi(3 \delta )}\left\|  y^{\frac{1-2s}{2}}_{n+1} w \right\|_{L^2(B_{3\delta})}^2 + e^{\tau\phi(\tilde{R}/4)} \delta^2 \tau^2 \left\|   y^{\frac{1-2s}{2}}_{n+1}w \right\|_{L^2(\tilde{R}/8,\tilde{R}/4)}^2\\
\lesssim \delta^2 e^{\tau \phi(\tilde{R}/2)} \left\|   y^{\frac{1-2s}{2}}_{n+1} w\right\|_{L^2(B_{2\tilde{R}})}^2 + e^{\tau \phi(\frac{\delta}{2})} \left\|   y^{\frac{1-2s}{2}}_{n+1} w\right\|_{L^2(B_{3 \delta / 2})}^2.
\end{align*}
Now, we choose $\tau>0$ such that $\delta^2 e^{\tau \phi(\tilde{R}/2)} \left\|   y^{\frac{1-2s}{2}}_{n+1} w\right\|_{L^2(B_{2\tilde{R}})}^2 $ on the right hand side can be absorbed in the term $e^{\tau\phi(\tilde{R}/4)}\tau^2\delta^2 \left\|   y^{\frac{1-2s}{2}}_{n+1}w \right\|_{L^2(\tilde{R}/8,\tilde{R}/4)}^2$ on the left hand side. A possible choice of $\tau$, for example, is
\begin{align*}
\tau \sim \frac{1}{\phi(\tilde{R}/2)- \phi(\tilde{R}/4)} \ln \left(  \frac{ \left\| y^{\frac{1-2s}{2}}_{n+1} w \right\|_{L^2(\tilde{R}/8,\tilde{R}/4)} }{  \left\|y^{\frac{1-2s}{2}}_{n+1} w  \right\|_{L^2(B_{2\tilde{R}})} } \right).
\end{align*}
This implies the doubling inequality for $r=\delta$ with a constant which, by virtue of the structure of $\phi$, does not depend on $\delta$. Since $0<\delta \ll \tilde{R} $ was arbitrary, this implies the doubling property. 
\end{proof}

\begin{rmk}
\label{rmk:doubl}
The more general claim of Remark \ref{rmk:doubl1} follows from two ingredients: a three balls inequality and an overlapping chains argument. The three balls inequality compares the value of $w$ on a ball of size $r$ with balls of size $\frac{r}{2}$ and $2r$:
\begin{align*}
\left\|y_{n+1}^{\frac{1-2s}{2}} w \right\|_{L^2(B_{r}(y_0)\cap \R^{n+1}_+)} \leq C \left\| y_{n+1}^{\frac{1-2s}{2}}w\right\|_{L^2(B_{\frac{r}{2}}(y_0)\cap \R^{n+1}_+)}^{\alpha} \left\| y_{n+1}^{\frac{1-2s}{2}} w \right\|_{L^2(B_{2r}(y_0)\cap \R^{n+1}_+)}^{1-\alpha},
\end{align*}
for sufficiently small radii $r>0$. This inequality allows to compare the values of $w$ along a chain of overlapping balls. Thus, it is possible to deduce an estimate of the form
\begin{align*}
\left\|y_{n+1}^{\frac{1-2s}{2}} w \right\|_{L^2(B_{r}(y_0))} \geq C_{r}\left\| y_{n+1}^{\frac{1-2s}{2}}w\right\|_{L^2(B_{2R}(z))}.
\end{align*}
Hence, the norms of $w$ on smaller balls can be related to the norms on the whole ball $B_{2R}(z)$. This then allows to deduce the stronger doubling inequality of Remark \ref{rmk:doubl1} as well as the reduction to sufficiently small balls in the proof of Proposition \ref{prop:doubl}. For further details we refer to the articles on quantitative unique continuation by Bakri \cite{Bakri}.
\end{rmk}

\subsection[Reduction to the WUCP]{Reduction to the Weak Unique Continuation Problem}
\label{sec:reduction}

In this section we explain how the previous estimates can be combined in order to reduce the strong unique continuation problem to its weak analogue. The key argument relies on a blow-up procedure.

\begin{prop}[SUCP II]
Let $s\in(0,1)$ and let $w: \R^{n+1} \rightarrow \R$, \\
$w \in H^{1}_{loc}(y_{n+1}^{1-2s}dy,\R^{n+1}_{+})$, be a solution of
\begin{align*}
\nabla \cdot  y_{n+1}^{1-2s} \nabla  w & = 0 \mbox{ in } \R^{n+1}_{+}, \\
-\lim\limits_{y_{n+1}\rightarrow 0} y^{1-2s}_{n+1} \p_{n+1} w &= V w \mbox{ on } \R^n,
\end{align*}
with $V= V_{1} + V_{2}$,
\begin{align*}
V_{1}(y) = |y|^{-2s}h\left( \frac{y}{|y|} \right),\ \ h\in L^{\infty},\ \ |V_{2}(y)| \leq c|y|^{-2s+\epsilon}. %, \ \ \left\| h \right\|_{L^{\infty}}\ll 1 .
\end{align*}
For $s < \frac{1}{2}$, we additionally require that one of the following assumptions is satisfied:
\begin{itemize}
\item the potential $V_2$ satisfies
$
V_{2}\in C^{1}(\R^n\setminus \{0\}) \mbox{ and }
|y\cdot \nabla V_{2}| \lesssim c|y|^{-2s+\epsilon},
$
\item $s\in [\frac{1}{4}, \frac{1}{2})$ and $V_1 \equiv 0$.
\end{itemize}
Suppose that $w(\cdot, 0)$ vanishes of infinite order at $0$. Then 
\begin{align*}
w \equiv 0.
\end{align*}
\end{prop}

\begin{proof}
Without loss of generality we may assume that $w$ does not vanish of infinite order in both the normal and tangential directions. We consider a rescaled version of $w$: Let $0<\sigma \ll 1$. We define 
\begin{align*}
w_{\sigma}(y) = \frac{w(\sigma y)}{\sigma^{-\frac{n+1}{2}} \sigma^{-\frac{1-2s}{2}} \left\| y_{n+1}^{\frac{1-2s}{2}} w\right\|_{L^2(B_{\sigma}^{+}(0))}}.
\end{align*}
Using the gradient estimate, we obtain 
\begin{align*}
\left\| y_{n+1}^{\frac{1-2s}{2}} \nabla w \right\|_{L^2(B_{\sigma}^{+})}^2
\lesssim &\ \frac{1}{\sigma^2}\left\| y_{n+1}^{\frac{1-2s}{2}}  w \right\|_{L^2(B_{2\sigma}^{+})}^2 + \int\limits_{ B_{2\sigma}^{+} \cap \{y_{n+1}=0\}} \eta^2 y^{1-2s}_{n+1} w \p_{n+1} w dy,\\
\lesssim &\ \frac{1}{\sigma^2}\left\| y_{n+1}^{\frac{1-2s}{2}}  w \right\|_{L^2(B_{\sigma}^{+})}^2 ,
\end{align*}
where the last line is a consequence of the doubling inequality as well as the finite order of vanishing of $w$ in the normal direction, c.f. Remark \ref{rmk:vani}: Due to the infinite order of vanishing, the boundary contributions can be absorbed in the other terms for sufficiently small $\sigma$. In effect, we have
\begin{itemize}
\item $\left\| y^{\frac{1-2s}{2}}_{n+1}w_{\sigma}\right\|_{L^2(B_{1}^{+})} = 1$,
\item $\left\| y^{\frac{1-2s}{2}}_{n+1} \nabla w_{\sigma}\right\|_{L^2(B_{1}^{+})} \leq C$.
\end{itemize} 
Hence, (along a not relabeled subsequence) we may pass to the limit $\sigma \rightarrow 0$ and obtain $w_{\sigma} \rightarrow w_{0}$ strongly in $L^{2}$ via Rellich's compactness theorem. As a consequence of the infinite order of vanishing (and the finite order of vanishing in the normal direction), $w_{\sigma}$ converges to zero on the boundary. Furthermore, $w_{0}$ weakly solves
\begin{align*}
\nabla \cdot y_{n+1}^{1-2s}\nabla w_{0}  & = 0 \mbox{ in } B_{1}^+(0),\\
\lim\limits_{y_{n+1}\rightarrow 0} y_{n+1}^{1-2s} \partial_{n+1} w_{0} & = 0 \mbox{ on } B_{1}^+(0)\cap \{y_{n+1}=0\}.
\end{align*}
Due to the weak unique continuation principle (c.f. Proposition \ref{prop:wuc}), $w_{0}$ has to vanish (which contradicts $\left\| y^{\frac{1-2s}{2}}_{n+1}w_{0}\right\|_{L^2(B_{1}^{+})} = 1$).
\end{proof}

\section{The One-Dimensional Situation}
\label{sec:1D}
In the case of \emph{one-dimensional} fractional Schr\"odinger equations it is possible to deduce stronger estimates than in the general case since the eigenvalues of the spherical contribution of the symmetric part of the operator satisfy a spectral gap condition. Moreover, they can be computed explicitly. For a fixed $s\in(0,1)$ the one-dimensionality of the problem reduces the eigenvalue equation to a one-parameter family of odes:
\begin{equation}
\begin{split}
\label{eq:eigenv}
\left(\frac{\partial^2}{\partial^2 \va} + \frac{(1-2s)(1+2s)}{4} \frac{1}{\sin(\va)^2} - \frac{(1-2s)^2}{4} \right) v &= \lambda v  \mbox{ in }  [0,\pi],\\
\lim\limits_{\sin(\va)\rightarrow 0} \sin(\va)^{1-2s} \dv \sin(\va)^{\frac{2s-1}{2}}v&=0 \mbox{ on }\{0,\pi\}.
\end{split}
\end{equation}
This can be reduced to generalized Legendre equations which allow to determine the admissible values of $\lambda$:

\begin{lem}
\label{lem:eigenv}
Let $s\in(0,1)$. Then the eigenvalues of (\ref{eq:eigenv}) are of the form
\begin{align*}
\lambda_{k}=  -\frac{(1-2s)^2}{4} - \left(k-s+\frac{1}{2}\right)^2, \ k\in \N_{\geq 0}.
\end{align*}
\end{lem}

Apart from the characterization of the eigenvalues, it is also possible to determine (some of) the associated eigenfunctions explicitly. This boils down to finding appropriate solutions of a generalized Legendre equation:

\begin{lem}
\label{lem:hyp}
Let $\nu = k-\mu$, $k\in \N_{\geq 0}$, $\mu\in(0,1)$.
Then the generalized Legendre equation
\begin{align}
\label{eq:Legeq}
(1-x^2)w''(x) - 2x w'(x) + \left( \nu(\nu+1) - \frac{\mu^2}{1-x^2} \right)w(x)=0,
\end{align}
has a solution of the form
\begin{align*}
f_{\nu}^{\mu}(x) = \frac{P_{k}(x)}{(1-x^2)^{\mu/2}},
\end{align*}
where $P_{k}(x)$ is a polynomial of degree (exactly) $k$. 
\end{lem}

\begin{proof}[Proof of Lemma \ref{lem:hyp}]
We consider solutions of the generalized Legendre equation (\ref{eq:Legeq}) for our choices of parameters $\mu$ and $\nu$. In order to solve the equation, we consider the ansatz
\begin{align*}
w(x)= \frac{P_{k}(x)}{(1-x^2)^{\mu/2}}.
\end{align*}
Inserting this into the generalized Legendre equation (\ref{eq:Legeq}), results in an equation for the $P_{k}$: 
\begin{align*}
(1-x^2)P''_{k}(x) + 2(\mu-1)xP'_{k}(x) + (k^2-2k\mu+k)P_{k}(x) =0.
\end{align*}
For a polynomial ansatz, $P_{k}(x)= \sum\limits_{j=0}^{k}\alpha_{j}x^j$, this turns into a recursion formula for the coefficients $\alpha_{j}$:
\begin{equation}
\begin{split}
\label{eq:coef}
2\alpha_{2} + (k^2-2k\mu+k)\alpha_{0}&=0,\\
6\alpha_{3} +  (2\mu-2+k^2-2k\mu+k)\alpha_{1} &= 0,\\
(j+1)(j+1)\alpha_{j+2} + (j(j-1)+2j(\mu-1)+k^2-2k\mu+k)\alpha_{j} & = 0, \mbox{ if } j\geq 2.
\end{split}
\end{equation}
This yields $k$ equations for the $k+1$ coefficients of the polynomial $P_{k}(x)$. Due to the restrictions $\mu \in (0,1)$ and $k\geq 0$, the (coefficient) equations can be solved explicitly if $k\leq 3$. Moreover, we notice that the equation
$$x(x-1)+2x(\mu-1)+k^2-2k\mu+k = 0, $$
has pairs of complex-valued solutions if $k\geq 4$ and $\mu \in (0,1)$ -- but no real ones. Hence, by the last equation in (\ref{eq:coef}), $a_{j+2}\neq 0$ if  $a_{j}\neq 0$. In effect, it is always possible to find a one-parameter family of solutions of system (\ref{eq:coef}). For even $k$ this depends on $a_{0}$, while for odd $k$ it depends on $a_{1}$. This proves the claim.
\end{proof}

\begin{proof}[Proof of Lemma \ref{lem:eigenv}]
The general (complex valued) solution of the ODE (\ref{eq:eigenv}) is given by
\begin{equation}
\begin{split}
\label{eq:reprfl}
v(\va)= & \ C_{1}(\cos^2(\va)-1)^{\frac{1}{4}}P_{\frac{1}{2}(-1+ \sqrt{-1-4\lambda + 4s - 4s^2})} ^{s}(\cos(\va))\\
&+C_{2}(\cos^2(\va)-1)^{\frac{1}{4}}Q_{\frac{1}{2}(-1+ \sqrt{-1-4\lambda + 4s - 4s^2})}^{s}(\cos(\va)),
\end{split}
\end{equation}
where $P_{\nu}^{\mu}(x)$ and $Q_{\nu}^{\mu}(x)$ are Legendre functions of the first and second kind, i.e. solutions of the generalized Legendre equation (\ref{eq:Legeq}).
In order to be an eigenfunction, the solution has to have vanishing generalized Neumann data. Setting
$\nu = k-s=\frac{1}{2}(-1+ \sqrt{-1-4\lambda + 4s - 4s^2})$, $k\in \N$, leads to simplifications: According to Lemma \ref{lem:hyp} there are solutions of the form
\begin{align*}
f_{\nu}^{\mu}(\cos(\va)) = \frac{P_{k}(\cos(\va))}{\sin(\va)^{s}},
\end{align*}
where $P_{k}(x)$ is a polynomial of degree $k$.
Thus, for this choice of $\nu$ the general solution (\ref{eq:reprfl}) becomes
\begin{align*}
v_{k}(\va)= \sin(\va)^{\frac{1-2s}{2}}P_{k}(\cos(\va)).
\end{align*}
Inserting this into the boundary condition, we infer that these functions do not only satisfy (\ref{eq:Legeq}) but also obey the right boundary conditions. Thus, these functions are indeed eigenfunctions of our equation. 
It remains to show that the corresponding eigenvalues constitute the whole spectrum, i.e. there are no further eigenvalues (which we might have missed by computing only special eigenfunctions). This follows from recurrence relations for the generalized Legendre functions. Setting $h^{\mu}_{\nu}(x)=c_{1}P^{\mu}_{\nu}(x)+c_{2}Q^{\mu}_{\nu}(x)$ with $c_{1},c_{2}\in \R$, we have (c.f. \cite{OLBC}):
\begin{align*}
\sin(\va)^{1-2s}\dv (\sin(\va)^{s}h_{\nu}^s(\cos(\va))) = & \ s(\sin(\va)^{s-1}\cos(\va)h^{s}_{\nu}(\cos(\va)) \\
&- (\sin(\va))^{s-1}[(s-\nu-1)h^{s}_{\nu +1}(\cos(\va)) \\
& + (\nu+1)\cos(\va)h^{s}_{\nu}(\cos(\va)) ]) \\
 = & \ -(\sin(\va))^{-s}[\cos(\va)(s-\nu-1)h^{s}_{\nu}(\cos(\va)) \\
&- (s-\nu-1)h_{\nu+1}^{s}(\cos(\va))].
\end{align*}
Due to the asymptotics of $Q_{\nu}^{\mu}(\cos(\va))$ at $\va=0$ (a symbolic Mathematica computation yields $Q_{\nu}^{\mu}(\cos(\va)) \sim \frac{2^{-s}\pi^2 1/ \sin(\pi s)1/ \sin(\pi(s+\nu))}{\Gamma(s)\Gamma(-s-\nu)\Gamma(1-s+\nu)}$), 
it follows that $c_{2}=0$ unless $\nu=k-s$ for $k\in \N_{\geq 0}$, as $P_{\nu}^{\mu}(\cos(\va))$ satisfies the boundary conditions at $\va=0$ for $\mu\in(0,1)$ and arbitrary $\nu$.
We claim that, in effect, only $\nu=k-s$ is admissible (in particular, none of the $P_{\nu}^{s}(\cos(\va))$ are admissible for $\nu\neq k-s$). This is a consequence of the connection formulas, c.f. \cite{OLBC}, for Legendre functions:
\begin{align*}
P_{\nu}^{\mu}(-x) = - \frac{2}{\pi}\sin((\nu+s)\pi)Q^{\mu}_{\nu}(x)+ \cos((\nu+s)\pi)P^{\mu}_{\nu}(x).
\end{align*}
Evaluated at $x=\cos(\pi)$, the asymptotics of $Q_{\nu}^{\mu}(\cos(\va))$ and of $P_{\nu}^{\mu}(\cos(\va))$ imply that $\nu=k-s$, $k\in \N$, is the only admissible family of parameters.
Thus, assuming the validity of the boundary conditions at $\va=0$ and at $\va=\pi$ necessarily leads to $\nu=k-s$, $k\in \N$. Combined with the form of $\nu$ given in (\ref{eq:reprfl}), this determines the possible eigenvalues.
\end{proof}

\begin{rmk}
The explicit representation of the eigenvalues illustrates that in the one-dimensional situation the spectral gap of the extension problem related to the fractional Laplacian is comparable with the spectral gap for the pure Laplacian (in that case $\lambda = -k^2$, $k\in \Z$).
\end{rmk}

The characterization of the spectrum of the one-dimensional Caffarelli extension allows to deduce stronger $L^2$ Carleman estimates similar to the ones in \cite{KT1}. In particular, it is possible to avoid the logarithmic loss in the Carleman estimate. As a consequence, it is possible to treat the strong unique continuation principle for potentials which are bounded by arbitrary scaling invariant Hardy type potentials: %with (degenerate) pseudoconvex weight functions \textit{without convexifying}: As in the paper of Koch and Tataru \cite{KT1}, it is now possible to use (almost) linear weight functions (in conformal coordinates) in the application of Proposition \ref{prop:KT1b}. Thus, we have

\begin{prop}
\label{prop:sg}
Let $s\in[\frac{1}{2},1)$ and let $w \in H^{s}(\R)$ be a solution of
\begin{align*}
(-\D)^{s}w = Vw \mbox{ in } \R.
\end{align*}
Assume that $w$ vanishes of infinite order at the origin and that $|V(y)| \lesssim |y|^{-2s}$ if $s> \frac{1}{2}$ and that $|V(y)| \leq c|y|^{-1}$ for $0<c\ll 1$ if $s=\frac{1}{2}$. %that $0< c \leq c_{0}$ is sufficiently small,  
Then $w\equiv 0$.
\end{prop}

\begin{proof}[Sketch of Proof]
The proof relies on strengthened Carleman bounds. In the case of a spectral gap, it is possible to give bounds which do not depend on the convexity parameter of the weight in exchange of a loss of half a power of $\tau$, c.f. Remark \ref{rmk:Carlsym}. Roughly speaking, in the $u$-coordinates, this results in a boundary estimate of the form
\begin{align*}
\tau^{\frac{2s-1}{2}}\left\| u \right\|_{L^2(\R \times \p S^{n}_+)} \lesssim \tau^{\frac{1-2s}{2}}\left\| e^{2st}V u\right\|_{L^2(\R \times \p S^{n}_+)} + \mbox{ bulk contributions}.
\end{align*}
This explains the slightly modified $s$-dependence of the estimate.
\end{proof}

%comment on size of the constant
%\begin{rmk}
%From the results of Sections \ref{sec:CarlI} and \ref{sec:CarlII} it seems plausible to expect $c_{0}\sim c_{HT}^{-1}$. However, this is not clear at the moment.  
%\end{rmk}

\section[Improved Integrability: The Half-Laplacian]{$L^p$-Regularity: Understanding the Half-Lapla-cian in the Framework of Koch \& Tataru}

\label{sec:Half}
As pointed out in the introduction, by an even reflection it is possible to interpret the unique continuation problem for the fractional Laplacian in the framework of Koch and Tataru \cite{KT1}. The potentials $W_{1}$ and $W_{2}$ are essentially given by $H(y_{n+1})V(y')$, with $H(y_{n+1})$ denoting a Heaviside function. The result of Koch and Tataru immediately demonstrates that for the half-Laplacian the strong unique continuation property holds with $V\in l^1_{w}(L^{n+1})$ under additional smallness assumptions as described in \cite{KT1}. For the half-Laplacian scaling arguments, however, suggest that the critical space is given by potentials $V\in L^{n}$ (possibly obeying some smallness assumption). Thus, it is natural to pose the question whether this can still be achieved in the framework of Koch and Tataru \cite{KT1}. As we briefly illustrate below, this is indeed possible for subcritical potentials:

\begin{prop}
Let $w\in H^{\frac{1}{2}}(\R^n)$ be a solution of
\begin{align*}
(-\D)^{\frac{1}{2}}w = V w \mbox{ in } \R^n.
\end{align*}
Assume that $V \in L^{n+\epsilon}(\R^n)$ and that $w$ vanishes of infinite order at the origin. Then $w\equiv 0$.
\end{prop}

\begin{proof}
The proof is based on a refined extension. 
We consider the following auxiliary problem: Let $\phi$ denote the harmonic (Neumann) extension of the potential $V$, i.e.
\begin{align*}
\D \phi & = 0 \mbox{ in } \R^{n+1}_{+},\\
\p_{n+1} \phi & = V \mbox{ on } \{y_{n+1}=0\}.
\end{align*}
Then by regularity of the elliptic Neumann problem
\begin{align*}
\phi \in W^{1+ \frac{1}{n+\epsilon}, n+\epsilon}(\R^{n+1}_{+}). 
\end{align*}
Hence, $\nabla \phi \in W^{\frac{1}{n+\epsilon},n+\epsilon}$ and by the Sobolev embedding theorem for Besov spaces (c.f. for example \cite{L}), we obtain $\nabla \phi \in L^{n+1+\delta}(\R^{n+1}_{+})$, with $\delta=\delta(\epsilon)$ being a continuous function in $\epsilon$ for sufficiently small $0\leq \epsilon \ll 1$ and satisfying $\delta\geq 0$, $\delta(0)=0$. This integrability property is preserved under an even reflection. With a slight abuse of notation the reflected solution then distributionally satisfies
\begin{align*}
\D \phi & = V \delta_{0}(y_{n+1}) \mbox{ in } \R^{n+1}.
\end{align*}
Reflecting the solution, $\tilde{w}$, of the Caffarelli extension of (\ref{eq:frac}) evenly and setting $W= \nabla \phi$, we infer
\begin{align*}
\D \tilde{w} = \nabla (W \tilde{w}) - W \nabla \tilde{w} \mbox{ in } \R^{n+1}.
\end{align*}
As the previous considerations imply that $W\in L^{n+1+\delta}(\R^{n+1})$, the result of Koch and Tataru can be applied. Their machinery then proves the claim.
\end{proof}

\begin{rmk}
This reduction to the Koch/Tataru setting suggests that the potential $V$ appearing in the equation for the half-Laplacian should be interpreted as a \textit{gradient} rather than a usual potential for an elliptic problem. In this case one cannot expect to deal with arbitrarily large potentials (in contrast to \cite{Pan}) as a counterexample by Wolff indicates \cite{W} (exactly scaling-critical potentials represent an exception).
\end{rmk}

\section[Variable Coefficients]{The Carleman Estimates for Variable Coefficient Operators}
\label{sec:vc}
In this final section on unique continuation properties of the fractional Laplacian we extend the previous results to operators with variable coefficients and operators on domains which are not half-spaces. The methods we present allow to deal with three situations:
\begin{itemize}
\item First, we restrict our attention to the flat half-space, $\R^{n+1}_+$, but consider a class of more general operators with non-constant metrics:
\begin{align*}
(\p_{n+1} y_{n+1}^{1-2s} \p_{n+1} + \nabla' \cdot y_{n+1}^{1-2s} a(y') \nabla')w &= 0 \mbox{ in } \R^{n+1}_+, \\
\lim\limits_{y_{n+1}\rightarrow 0} y_{n+1}^{1-2s} \p_{n+1} w &= Vw \mbox{ on } \R^{n}.
\end{align*}
Here $a(y')$ is a tensor which satisfies certain Lipschitz bounds. We note that, in particular, this situation corresponds to generalizations of the Caffarelli-Silvestre extension for variable coefficients. Thus, it is possible to think of the results on these operators as statements on ``variable coefficient" fractional Laplacians.
\item In the second case, we study the analogous situation on manifolds with sufficiently regular boundaries. As we are only interested in a local statement, we consider the situation in local coordinates in a coordinate patch:
\begin{equation}
\begin{split}
\label{vc:boundary}
(\p_{\nu} d_{\p \Omega}(y)^{1-2s} \p_{\nu} + \nabla_{tan} \cdot d_{\p \Omega}(y)^{1-2s} a(y_{tan}) \nabla_{tan}) w & = 0 \mbox{ in } \Omega,\\
\lim\limits_{d_{\p \Omega}(y)\rightarrow 0} d_{\p \Omega}(y)^{1-2s} \p_{\nu}w = Vw \mbox{ on } \partial \Omega.
\end{split}
\end{equation}
In this context we use $\p_{\nu}$ to denote the ``normal" and $\nabla_{tan}$ the ``tangential" derivatives in appropriate normal coordinates; $d_{\p \Omega}(y)$ represents the distance function with respect to the boundary. This setting can be treated in analogy to the flat situation (here we emphasize that first order contributions which originate from the global formulation via corresponding Laplace Beltrami operators on the manifold represent controllable errors, c.f. step 4 in the proof of Proposition \ref{prop:Hardypot}). As before, the equation can be interpreted as a generalization of the Caffarelli-Silvestre extension to domains with non-flat boundary.
\item Last but not least, we comment on the half-Laplacian and the one-dimensional situation for which stronger results are available due to the presence of the already discussed spectral gap. As a consequence, perturbation techniques as in \cite{KT1} are available.
\end{itemize}
Since the second situation can be reduced to the first one, we emphasize the details in the $\R^{n+1}_+$-case and only point out the modifications in the second situation.

\subsection{The Half-Space Situation with Variable Coefficients and Differentiability}

In this section we address the half-space situation with variable coefficients. In this context, we use the following conventions and notations, c.f. \cite{Jo}: 
\begin{itemize}
\item Let $(M,g)$ be a Riemannian manifold of dimension $m$, assume that $p\in M$, $v\in T_{p}M$ and let $c_{v}:[0,\epsilon]\rightarrow M$ be a geodesic with $c_{v}(0)=p$, $\dot{c}_{v}(0)=v$. Set $V_{p}:= \{v\in T_{p}M| \ c_{v} \mbox{ is defined on } [0,1]\}$. Then we define
\begin{align*}
\exp_{p}:V_{p} \rightarrow M, \ \
v \mapsto c_{v}(1).
\end{align*}
If we want to point out the dependence on the metric, we also use the notation $\exp_{g,p}$. We remark that if $T_{p}M$ is identified with $\R^m$ the exponential map yields a local choice of coordinates. 
\item Let $(M,g)=(\R \times M, 1 \times g(y'))$. We set $\la_{g}(y):= \sqrt{y_{n+1}^2 + \lb_{g}(y')^2}$ with $y=(y',y_{n+1})$ and $\lb_{g}(y')$ being the geodesic distance of $y'$ from the origin with respect to the metric $g(y')$ on $\R^n$.
%\item Let $(M,g)=(\R^m,g)$ be a Riemannian manifold. Then we use the symbol $|t_{g(y)}|$ to represent the (matrix) norm of the inverse Jacobi matrix associated with the change of coordinates $z=\exp^{-1}_{g,0}(y)$.
\end{itemize}
With this, we can prove the following Proposition:

\begin{prop}[Variable Coefficient Carleman Estimate]
\label{prop:Hardypot}
Suppose that $a:\R^{n}\rightarrow \R^{n\times n}$ with
\begin{align*}
\lambda |\xi|^2 \leq \xi\cdot a(y') \xi \leq \Lambda |\xi|^2, \ \ 0 < \lambda \leq \Lambda < \infty, \ \ a_{ij}= a_{ji}, \ \
a\in C^{2}. %\ \ |\la_{a^{-1}}(y)||\nabla a^{-1}(y') | \lesssim \epsilon (1+ \ln(\la_{a^{-1}}(y'))^{2})^{-\frac{1}{2}}, % estimate for \frac{dx}{d(exp_{a^{-1}})} needs to be included!
\end{align*} 
%with $0<\epsilon \ll 1$. 
%Let $\la_{a^{-1}}(y):= \sqrt{y_{n+1}^2 + \lb_{a^{-1}}(y')^2}$ with $y=(y',y_{n+1})$ and $\lb_{a^{-1}}(y')$ being the geodesic distance of $y'$ from the origin with respect to the metric $a(y')^{-1}$ on $\R^n$. 
Let $s\in[\frac{1}{4},1)$ and set 
$$\phi(y)= - \ln(\la_{a^{-1}}(y)) + \frac{1}{10}\left( \ln(\la_{a^{-1}}(y))\arctan(\la_{a^{-1}}(y)) - \frac{1}{2} \ln(1+\ln(\la_{a^{-1}}(y))^2) \right).$$
Assume that $w\in H^{1}(y_{n+1}^{1-2s}dy,{\R^{n+1}_{+}})$ with $\supp{(w)} \subset \overline{B_{r}(0)^{+}}$, $0<r=r(a)\ll 1$, satisfies
\begin{align*}
(\p_{n+1} y_{n+1}^{1-2s} \p_{n+1} + \nabla' \cdot y_{n+1}^{1-2s} a(y') \nabla')w &= f \mbox{ in } \R^{n+1}_+, \\
\lim\limits_{y_{n+1}\rightarrow 0} y_{n+1}^{1-2s} \p_{n+1} w &= Vw \mbox{ on } \R^{n},
\end{align*}
and vanishes of infinite order at $0$.
Further assume that $V= V_{1} + V_{2}$,
\begin{align*}
&V_{1}(y) = \la_{a^{-1}}(y)^{-2s}h\left(\frac{y}{\la_{a^{-1}}(y)}\right), \ \  h\in L^{\infty},\ \
|V_{2}(y)| \leq c \la_{a^{-1}}(y)^{-2s+\epsilon},\\
&V_{2}(y)\in C^{1}(\R^n \setminus \{0\}), \ \ |\nabla V_{2}(y)| \leq \la_{a^{-1}}(y)^{-2s+\epsilon-1}.
\end{align*} %conditions on $a$
Then for $\tau\geq \tau_{0}>0$ we have
\begin{align*}
&\tau^{s}\left\| e^{\tau \phi}(1+\ln(\la_{a^{-1}}(y))^2)^{-\frac{1}{2}} \la_{a^{-1}}(y)^{-s} w \right\|_{L^2(\R^{n})}\\ 
&+ \tau \left\|  e^{\tau \phi} (1+\ln(\la_{a^{-1}}(y))^2)^{-\frac{1}{2}} \la_{a^{-1}}(y)^{-1}  y_{n+1}^{\frac{1-2s}{2}} w \right\|_{L^2(\R^{n+1}_+)}^2 \\
&+  \left\| e^{\tau \phi} (1+\ln(\la_{a^{-1}}(y))^2)^{-\frac{1}{2}} y_{n+1}^{\frac{1-2s}{2}} \nabla w \right\|_{L^2(\R^{n+1}_+)}^2\\
\lesssim & \  \tau^{-\frac{1}{2}}\left\|e^{\tau \phi} \la_{a^{-1}} (y) y_{n+1}^{\frac{2s-1}{2}}f \right\|_{L^2(\R^{n+1}_+)}
 + \tau^{\frac{1-2s}{2}}\left\| e^{\tau \phi} \la_{a^{-1}}(y)^{s} Vw \right\|_{L^2(\R^{n})}.
\end{align*}
\end{prop}

\begin{rmk}
\begin{itemize}
\item The $C^2$ regularity condition on the metric is an artifact of our strategy of proof: We make use of the exponential map associated with the metric $a^{-1}(y')$ in order to pass to geodesic polar coordinates. An alternative strategy using arguments from \cite{KT1} would have been possible. With this method it is possible to reduce to the (optimal) setting of Lipschitz metrics.
\item The radius $r>0$ in the proposition is chosen so small that we may pass to geodesic normal coordinates in it. This is no restriction in general, as it is possible to use appropriate cut-off functions.
%\item The Lipschitz condition involving the Jacobian of an appropriate coordinate change originates from the reduction of the variable coefficient situation to a situation which is comparable to the constant coefficient case from above. This requires passing over to normal coordinates in which geodesics with respect to the metric $a^{-1}(y')$ become rays. 
\item We use the notation $a(y')^{-1}$ to denote the pointwise inverse of $a(y')$, i.e. $$a(y')^{-1}a(y')= \delta_{ij}.$$
\end{itemize}
\end{rmk}

In order to prove the desired Carleman inequality, we carry out a change of coordinates similar to the one described in the article of Koch and Tataru \cite{KT1}. Working with variable metrics, we have to introduce appropriate normal coordinates first. Thus, we cast our equation into a Riemannian framework where the Riemannian metric $g$ is given by $a^{-1}$. We note that after the change of coordinates our argument strongly resembles the proof of Corollary \ref{cor:schp} in the case of .

\begin{proof}[Proof of Proposition \ref{prop:Hardypot}]
Step 1: Choice of Coordinates. 
We cast the equation into a Riemannian framework. In this context we may interpret the tangential part of the operator as
\begin{align*}
\nabla' \cdot a(y') \nabla' = \D_{a^{-1}}' - \frac{1}{2}v_{a^{-1}}(y')\cdot a(y') \nabla' = \D'_{a^{-1}} - \frac{1}{2}v_{a^{-1}}(y')\cdot \nabla_{a^{-1}}',
\end{align*}
where $v_{a^{-1}}(y')$ is a vector with $i$-th component given by $v_{a^{-1},i}(y')= \tr (a^{-1}(y') \frac{\p a}{\p y_{i}})$. Here $\D_{a^{-1}}'$ and $\nabla_{a^{-1}}'$ denote the Laplace-Beltrami and gradient operators with respect to the metric $a(y')^{-1}$. We point out that the thus introduced metric is truly Riemannian as -- due to the $y'$-dependence of $a$ -- it depends on the point of evaluation. For the moment, we ignore the first order contribution in the definition of our operator. It can be considered as ``small" and can be treated as a controlled error contribution.\\
With this interpretation of the tangential operator, the full operator can be interpreted as a (degenerate) elliptic operator acting on the Riemannian manifold $(\R_{+}\times \R^{n}, 1\times a(y')^{-1})$.
In this setting, we aim at reducing the situation to geodesic polar coordinates. These can be obtained by first introducing Riemannian normal coordinates in the tangential directions and then passing to (geodesic) polar coordinates in the tangential and normal variables. \\
We commence by considering the tangential geometry: We may interpret it as the manifold $(\R^{n},a_{ij}(y')^{-1})$. Using (the locally well-defined) exponential map, we obtain normal coordinates on an open subset of $\R^n$ (here we make use of the $C^2$ condition on the metric $g$). As our Carleman estimates are formulated as local estimates for functions which are supported sufficiently close to zero, we assume that the change of coordinates is a global one and that our new manifold is given by $(\R^n,\bar{g}_{ij})$. This change of coordinates straightens out the geodesics passing through the origin.\\%We remark that the metric $\bar{g}_{ij}$ inherits the Lipschitz regularity and uniform ellipticity of $a_{ij}$ (with a possibly modified value of its norm and ellipticity constants). \\
Now we consider the full operator in the whole of $(\R_+ \times \R^n, 1 \times  \bar{g}_{ij})$ and introduce polar, instead of Cartesian coordinates in $\R^{n+1}_+$. This leads to a new spherical metric $g_{\theta \theta}$ and to a modified operator:
\begin{align*}
&\theta_{n}^{1-2s}\frac{1}{r^n}\drr(r^{n+1-2s}\drr) + \theta_{n}^{1-2s} r^{-1-2s} \frac{1}{2} \tr(g_{\theta \theta} \drr g_{\theta \theta}^{-1}) \drr\\
&+ r^{-1-2s}\frac{1}{\sqrt{\det g_{\theta \theta}}}\p_{\theta_i}\cdot \theta_{n}^{1-2s}g_{\theta\theta}^{-1}(r,\theta)\sqrt{\det{g_{\theta \theta}}}\p_{\theta_{j}}.
\end{align*}
%Using the bounds on $a$, we treat the first order contribution as a controlled error (c.f. step 4) and rewrite the (non-standard) spherical Laplacian as
%\begin{align*}
%\frac{1}{\sqrt{\det g_{\theta \theta}}}\p_{\theta_i}\cdot \theta_{n}^{1-2s}g_{\theta\theta}^{-1}(r,\theta)\sqrt{\det{g_{\theta \theta}}}\p_{\theta_{j}} = \nabla_{\theta}\cdot \theta_n^{1-2s}g_{\theta \theta}^{-1}(r,\theta) \nabla_{\theta} \\ +
%\frac{1}{\sqrt{\det{g_{\theta \theta}}}}( \nabla_{\theta} \sqrt{\det{g_{\theta \theta}}} ) \cdot \theta_{n}^{1-2s} g_{\theta \theta}^{-1}(r,\theta) \nabla_{\theta}
%\end{align*}
%Again, we only focus on the second order term, while the first oder term can be considered as a controlled error contribution. 
In the sequel, we will also denote the spherical metric $g_{\theta\theta}(r,\theta)$ by $g(r,\theta)$ and ignore the first order term involving the derivatives of $g_{\theta \theta}$. Due to the smallness of the homogeneous Lipschitz norm of $g$, it can be treated as a controlled error contribution which can be absorbed in the positive bulk terms.\\
We carry out the change into conformal coordinates, i.e. $r=e^{t}$, which yields $\drr = e^{-t}\dt$. This results in
\begin{align*}
e^{-(1+2s)t}\left[ \theta_{n}^{1-2s} \dt^2 + \left(n-2s \right)\theta_{n}^{1-2s}\dt + \tilde{\nabla}_{S^{n}}\cdot \theta_{n}^{1-2s}  \tilde{\nabla}_{S^{n}} \right],
\end{align*}
where for brevity of notation we used $\tilde{\nabla}_{S^n}$ to denote the spherical gradient with respect to our (non-standard) spherical metric. Conjugating with $e^{-\frac{n-2s}{2}t}$ (which corresponds to setting $w=e^{- \frac{n-2s}{2}t}u$) and multiplying the operator with $e^{(1+2s)t}$, results in
\begin{align*}
\theta_{n}^{1-2s}\left(\dt^2 - \frac{(n-2s)^2}{4}\right) + \tilde{\nabla}_{S^{n}}\cdot \theta_{n}^{1-2s} \tilde{\nabla}_{S^{n}}.
\end{align*}
Due to the product structure of our original manifold, the boundary condition turns into
$\lim\limits_{\theta_{n}\rightarrow 0} \theta_n^{1-2s} \p_{\va_{n}} u = e^{2st}V u.$
In analogy to the flat case and with a slight abuse of notation, we use the symbol $d \theta$ to denote the volume form of our (non-standard) spherical metric. In the sequel all the integrals will be computed with respect to this volume form.\\

Step 2: Computing the Commutator.
In order to separate the spherical and the radial variables, we set $u = \theta^{\frac{2s-1}{2}}_{n}v$ and multiply with $\theta_{n}^{\frac{2s-1}{2}}$. Although the function $v$ becomes increasingly singular (if $s>\frac{1}{2}$), this form of the equation has the advantage that the operator is symmetric and strictly separates the radial and spherical variables. Thus -- up to the first order error terms originating from the first step  -- our equation turns into
\begin{align*}
\dt^2 - \frac{(n-2s)^2}{4}+ \theta^{\frac{2s-1}{2}}_{n}\tilde{\nabla}_{S^{n}}\cdot \theta_{n}^{1-2s} \tilde{\nabla}_{S^{n}}\theta^{\frac{2s-1}{2}}_{n}.
\end{align*}
Conjugation with an only $t$-dependent weight, $\phi$, leads to the following ``symmetric and antisymmetric'' parts of the operator:
\begin{equation*}
\begin{split}
S & = \dt^2 + \tau^2(\dt \phi)^2 - \frac{(n-2s)^2}{4} + \theta^{\frac{2s-1}{2}}_{n}\tilde{\nabla}_{S^{n}}\cdot \theta_{n}^{1-2s}\tilde{\nabla}_{S^{n}} \theta^{\frac{2s-1}{2}}_{n},\\
A & = -2\tau (\dt \phi)\dt - \tau \dt^2 \phi.
\end{split}
\end{equation*}
We point out that the $\dt$-contributions are not actually symmetric and antisymmetric with respect to our non-standard spherical metric, yet this separation of the full operator into $S$ and $A$ proves to be convenient for the calculations of the pairing $(Su,Au)_{L^2(S^n_+\times \R)}$. All the occurring error terms can be controlled.
If $\phi$ is sufficiently pseudoconvex the separation into $S$ and $A$ yields the following ``commutator'' terms:
\begin{align*}
&4\tau^3 \left\| (\dt^2 \phi)^{\frac{1}{2}} \dt \phi v \right\|_{L^2( S^n_+ \times \R)}
+ 4\tau \left\| (\dt^2 \phi)^{\frac{1}{2}} \dt v \right\|_{L^2( S^n_+ \times \R)}
- \tau \int\limits_{ S^{n}_+ \times \R} \dt^{4} \phi v^2 d\theta dt
\\
&+ \mbox{(ER)},
\end{align*}
where (ER) is used to denote any bulk term involving derivatives of $g$ which is controlled by
\begin{equation}
\begin{split}
\label{eq:variableerr}
\tau \int\limits_{ S^{n}_+ \times \R} |\dt \phi| |\tilde{\nabla} v|^2 |\nabla g| d\theta dt.
\end{split}
\end{equation}
We remark that all integrals are calculated with respect to our non-standard spherical metric.
In these calculations one has to be slightly more careful than in the case of the standard sphere as the metric tensor, and thus the volume element, also depends on the $t$-variable. As a consequence, it is more convenient to calculate some of the quantities appearing in $(Su,Au)_{L^2_{g}(S^n_+ \times \R)}$ directly, instead of symmetrizing and antisymmetrizing the respective contributions. 
Contributions of the form (ER) will be treated as errors, c.f. Step 4.\\
Furthermore, weighted gradient estimates can be obtained:
\begin{equation}
\label{eq:vcoor}
\begin{split}
& ((\dt^2 \phi) \dt v, \dt v) + ((\dt^2 \phi)\theta^{1-2s}_{n} \tilde{\nabla}_{S^{n}}\theta_{n}^{-\frac{1-2s}{2}}v,\tilde{\nabla}_{S^{n}}\theta_{n}^{-\frac{1-2s}{2}}v)\\
=& \ -(Sv,(\dt^2 \phi) v) + \int\limits_{\p S^{n}_{+}\times \R} (\dt^2 \phi)(\theta^{1-2s}_{n} \nu \cdot \tilde{\nabla}_{S^{n}}\theta^{- \frac{1-2s}{2}}_{n}v) \theta^{- \frac{1-2s}{2}}_{n}v d\theta dt \\
& + ((\dt^4 \phi) v, v) + \mbox{(ER)}\\
 \leq & \ \frac{1}{2 \tau^2 } \left\| Sv \right\|_{L^2}^2 + \frac{\tau^2}{2} \left\| (\dt^2 \phi) v\right\|_{L^2}^2
+ \tau^2 \left\| (\dt^2 \phi)^{\frac{1}{2}}\dt \phi v \right\|_{L^2}^2 - \frac{(n-2)^2}{4} \left\| (\dt^2 \phi)^{\frac{1}{2}} v \right\|_{L^2}^2\\
&+\int\limits_{\p S^{n}_{+}\times \R} (\dt^2 \phi) (\theta^{1-2s}_{n} \nu \cdot \tilde{\nabla}_{S^{n}}\theta^{- \frac{1-2s}{2}}_{n}v) \theta^{- \frac{1-2s}{2}}_{n}v d\theta dt\\
& + ( (\dt^4 \phi) v, v) + \mbox{(ER)},
\end{split}
\end{equation}
where $\nu = (0,...,0,-1)$ denotes the outer unit normal.
For sufficiently pseudoconvex weight, $\phi$, the right hand side can even be controlled via the commutator contributions if everything is multiplied by a factor of $c\tau$, for example $c \sim \frac{1}{2}$ would work. 
The boundary integral can be evaluated to yield
\begin{align*}
\int\limits_{\p S^{n}_{+}\times \R}  (\dt^2 \phi) (\theta^{1-2s}_{n}\nu \cdot \tilde{\nabla}_{S^{n}}\theta^{- \frac{1-2s}{2}}_{n}v) \theta^{- \frac{1-2s}{2}}_{n}v d\theta dt\\
 = \int\limits_{\p S^{n}_{+}\times \R} \theta_{n}^{-(1-2s)} (\dt^2 \phi) e^{2st} Vv^2 d\theta dt,
\end{align*} 
where by a slight abuse of notation we also denote the lower dimensional volume form by $d\theta dt$.
We note that the gradient contribution in (\ref{eq:vcoor}) (multiplied with $\tau$) in particular suffices to absorb the bulk contribution of (\ref{eq:variableerr}).\\
The remaining boundary integral which originates from the commutator calculation is given by
\begin{align*}
&4\tau \int\limits_{\p S^{n}_{+}\times \R} (\theta^{1-2s}_{n} \nu \cdot  \tilde{\nabla}_{S^{n}}\theta^{- \frac{1-2s}{2}}_{n}v) (\dt \phi) \theta^{- \frac{1-2s}{2}}_{n}\dt v d\theta dt + \mbox{(BER)} \\
&+ 2\tau \int\limits_{\p S^{n}_{+}\times \R} ( \theta^{1-2s}_{n} \nu \cdot  \tilde{\nabla}_{S^{n}}\theta^{- \frac{1-2s}{2}}_{n}v) (\dt^2 \phi) \theta^{- \frac{1-2s}{2}}_{n} v d\theta dt \\
= & \ 4\tau \int\limits_{\p S^{n}_{+}\times \R} (\dt \phi) \theta_{n}^{-(1-2s)} e^{2st}V v \dt v d\theta dt\\
&+ \ 2\tau \int\limits_{\p S^{n}_{+}\times \R} (\dt^2 \phi) \theta_{n}^{-(1-2s)}  e^{2st}V v^2 d\theta dt + \mbox{(BER)},
\end{align*}
where (BER) denotes boundary contributions involving derivatives of the metric, e.g. terms bounded by $ \tau \int\limits_{\p S^{n}_+ \times \R} |\dt \phi \dt \bar{g} e^{2st} V| u^2 d\theta dt$.
Rewritten in terms of $u = \theta_{n}^{-\frac{1-2s}{2}} v$ the Carleman estimate reads
\begin{equation}
\label{vc:Carl}
\begin{split}
&c\tau \left\| (\dt^2 \phi)^{\frac{1}{2}} \theta^{\frac{1-2s}{2}}_{n} \dt u \right\|_{L^2}^2
+ c\tau \left\| (\dt^2 \phi)^{\frac{1}{2}} \theta^{\frac{1-2s}{2}}_{n} \tilde{\nabla}_{S^{n}} u \right\|_{L^2}^2\\
&+ c\tau^3  \left\| \theta^{\frac{1-2s}{2}}_{n}(\dt^2 \phi)^{\frac{1}{2}}(\dt \phi) u \right\|_{L^2}^2\\
&+ \left\| S (\theta^{\frac{1-2s}{2}}_{n} u) \right\|_{L^2}^2
+ \tau^{-1} \left\| (\dt^2 + \theta^{-\frac{1-2s}{2}}_{n} \tilde{\nabla}_{S^{n}}\cdot \theta_{n}^{1-2s} \tilde{\nabla}_{S^n}  \theta^{- \frac{1-2s}{2}}_{n}) \theta^{ \frac{1-2s}{2}}_{n} u \right\|_{L^2}^2\\
& + 4\tau \int\limits_{\p S^{n}_{+}\times \R} (\dt \phi) e^{2st} V  u \dt u d\theta dt
+ 2\tau \int\limits_{\p S^{n}_{+}\times \R} (\dt^2 \phi) e^{2st} V  u^2 d\theta dt\\
&+ c\tau \int\limits_{\p S^{n}_{+}\times \R}(\dt^2 \phi) e^{2st} V  u^2 d\theta dt
+ \mbox{(BER)}\\
\leq & \ \left\| L_{\phi} u \right\|_{L^2}^2,
\end{split}
\end{equation}
where
\begin{align*}
L_{\phi} = & \ \theta^{ \frac{1-2s}{2}}_{n}(\dt^2 + \tau^2 (\dt \phi)^2 - \frac{(n-2s)^2}{4} - 2\tau(\dt \phi)\dt - \tau \dt^2 \phi) \\
&+ \theta^{\frac{2s-1}{2}}_{n} \tilde{\nabla}_{S^{n}}\cdot \theta_{n}^{1-2s} \tilde{\nabla}_{S^n}.
\end{align*}
It remains to discuss the unsigned boundary contributions and the error terms. \\

Step 3: Bounding the Boundary Contributions.
In order to estimate the unsigned boundary contributions from the previous steps, we consider the respective expressions in polar coordinates as in (\ref{vc:Carl}). Starting with the scaling-critical potentials, we have to bound
\begin{align*}
& 4\tau \int\limits_{\p S^{n}_{+}\times \R}  (\dt \phi) e^{2st} V_{1} u \dt u d\theta dt
+ 2\tau \int\limits_{\p S^{n}_{+}\times \R}  (\dt^2 \phi) e^{2st} V_{1} u^2 d\theta dt\\
&+ c\tau \int\limits_{\p S^{n}_{+}\times \R}  (\dt^2 \phi) e^{2st} V_{1} u^2 d\theta dt
+ \tau \int\limits_{\p S^{n}_+ \times \R} |\dt \phi (\dt g) e^{2st} V_{1}| u^2 d\theta dt,
\end{align*}
i.e. we have to control the boundary integrals involving the potential $V_{1}=e^{-2st}h(\theta)$. By an integration by parts in $t$, we obtain that most contributions drop out. Indeed, the conditions on $a$ imply that the only non-vanishing terms can be estimated by
% or derivatives fall on $\eta_{r}$(as can be seen by integration by parts). In fact the this integrations by parts strategy leaves 
$
C\tau \int\limits_{\p S^{n}_{+}\times \R}|\dt^2 \phi| |h(\theta)| u^2 d\theta dt.
$
However, by appealing to the interpolation inequality (\ref{eq:trint}), this can be controlled by the positive quantities of the Carleman inequality: 
\begin{align*}
\tau \int\limits_{\p S^{n}_{+}\times \R}(\dt^2 \phi) u^2 d\theta dt \leq & \
\tau^{1-2s} \left\| (\dt^2 \phi)^{\frac{1}{2}} \theta^{\frac{1-2s}{2}}_{n} \nabla u \right\|_{L^2(S^n_+ \times \R)}^2\\
&+ \tau^{3-2s} \left\| (\dt^2 \phi)^{\frac{1}{2}} \theta^{\frac{1-2s}{2}}_{n}  u \right\|_{L^2(S^n_+ \times \R)}^2,
\end{align*}
where $\nabla =(\dt, \tilde{\nabla}_{S^n})$. Here we also used the explicit expression of $\phi$ and the support condition on $u$.\\
All the remaining boundary contributions involve the potential $V_{2}$ which has subcritical growth at zero. Due to the form of $\phi$, it suffices to deduce control of the term
\begin{align*}
4\tau \int\limits_{\p S^{n}_{+}\times \R} (\dt \phi) e^{2st} V_{2} u \dt u d\theta dt.
\end{align*}
Integrating by parts in $t$, using the subcriticality of $V_{2}$ and the properties of $\phi$ and $a$, it suffices to bound
\begin{align*}
C\tau \int\limits_{\p S^{n}_{+}\times \R} e^{\epsilon t}  u^2 d\theta dt.
\end{align*}
Again, this can be controlled by the interpolation inequality (\ref{eq:trint}).\\

Step 4: Treatment of the Error Contributions. It remains to comment on the first order error terms from step 1 and from the conjugation process. The terms from step 1 also undergo the conjugation process. Under this they either remain unchanged or involve a derivative of $\phi$ and a prefactor of $\tau$. Instead of including these contributions -- which, in the following, we denote by (Er) -- in the commutator calculation, we treat them as errors:
\begin{align*}
\left\| e^{\tau \phi} L w \right\|_{L^2} = \left\| (S + A + Er)u \right\|_{L^2}
\geq \left\| (S+A)u \right\|_{L^2} - \left\| (Er)u\right\|_{L^2}.
\end{align*}
Due to the assumptions on $a$ and the fact that these terms are only of first order, it is possible to absorb these specific errors -- as well as any error terms of the form (ER) -- into the positive commutator contributions which were deduced in step 3.
\end{proof}

%Combined with an appropriate weak unique continuation statement -- which for example holds for $a(y')\in C^{\infty}$ with controlled derivatives -- the strong unique continuation property also follows for these operators. 
Combined with a blow-up procedure comparable to the one carried out in Section \ref{sec:reduction}, the strong unique continuation property follows. As the blown-up solution satisfies an equation with \emph{constant} coefficients, the strong unique continuation result can be regarded as a consequence of a weak unique continuation statement of the same flavour as the one presented in Section \ref{sec:weak}.

\subsection{Carleman Inequalities in the Case of Non-Flat Domains}
The previous discussion of the situation in $\R^{n+1}_+$ illustrates that it is possible to deal with our (degenerate) elliptic operators if they are defined on a manifold of the form $(\R_{+}\times M, 1\times  g_{ij})$. Here $(M,g_{ij})$ is a Riemannian manifold which has -- in a Lipschitz sense -- a metric which is sufficiently close to a constant non-degenerate metric. In particular, this allows to deal with operators of the form (\ref{vc:boundary}) -- i.e. operators in which there is a clear distinction between normal and tangential variables, as, sufficiently close to the boundary, an appropriate choice of normal coordinates allows to cast the equation into (a lower order perturbation of) the previously discussed setting of $\R^{n+1}_+$.

\subsection{Comments on the Situation with a Spectral Gap}
In settings involving a spectral gap, the situation improves significantly. In fact, under these assumptions it is possible to argue as in the article of Koch and Tataru \cite{KT1} in which a radial summability condition is required -- which is based on stronger estimates originating from a spectral gap condition. Thus, in situations involving a spectral gap, it is possible to control equations with leading order contributions of the form
$\p_{i}  y_{n+1}^{1-2s} g_{ij}(y) \p_{j}$ and bounds of the type $|y||\nabla g|\in l^{1}(L^{\infty})$, c.f. \cite{KT1}.

\bibliography{citations}
\bibliographystyle{alpha}

\end{document}